\documentclass[12pt,oneside,a4paper]{amsart}
\usepackage{url}
\usepackage{graphicx}
\usepackage{subcaption}
\usepackage{a4}
\usepackage{amssymb}
\usepackage{amsbsy}
\usepackage{amsmath}
\usepackage{amstext}
\usepackage{amsopn}
\usepackage{amsthm}
\usepackage{mathtools}
\newtheorem{thm}{Theorem}[section]

\newtheorem{pro}[thm]{Proposition} 
\newtheorem{lem}[thm]{Lemma}

\newtheorem{cor}[thm]{Corollary} 
\theoremstyle{definition}
\theoremstyle{definition}
 
\newtheorem{de}[thm]{Definition}

\newtheoremstyle{uebstyle}{}{}{\footnotesize}{}{\normalsize\scshape}{}{\newline}{}
\theoremstyle{uebstyle}
\newtheorem{uebungen}[thm]{\"Ubungsaufgaben}
\newtheorem{loesungen}[thm]{L\"osungen zu \"Ubungsaufgaben}
\numberwithin{equation}{section}

\newcommand{\mycomment}[1]{}

\newcommand{\ob}[1]{\mathbb{#1}}
\newcommand{\gb}[1]{#1}
\DeclareMathAlphabet{\mathsc}{OT1}{cmr}{m}{sc}
\DeclareMathAlphabet{\mathbfsl}{OT1}{cmr}{bx}{it}
\newcommand{\vb}[1]{\mathbfsl{#1}}

\newcommand{\N}{\mathbb{ N}}

\newcommand{\C}{\mathbb{ C}}

\newcommand{\ul}[1]{\underline{#1}}

\newcommand{\LT}{\mathsc{Lt}}

\newcommand{\LM}{\mathsc{Lm}}

\newcommand{\LEXP}{\mathsc{Lexp}}

\newcommand{\LC}{\mathsc{Lc}}

\newcommand{\lea}{\le_a}

\newcommand{\md}{\operatorname{mdeg}}
\newcommand{\NON}{\N_0^n}

\newcommand{\Supp}{\operatorname{Supp}}
\newcommand{\Mon}{\operatorname{Mon}}

\newcommand{\kxn}{\ob{K}[x_1,\ldots,x_n]}

\newcommand{\kn}{\ob{K}^n}
\newcommand{\kk}{\ob{K}}

\newcommand{\submod}[1]{\langle #1 \rangle}

\newcommand{\mono}[1]{\vb{x}^{\gb{#1}}}

\newcommand{\lcm}{\operatorname{lcm}}

\newcommand{\V}{\mathbb{V}}
\newcommand{\I}{\mathbb{I}}

\newcommand{\epsi}{\varepsilon}

\newcommand{\XI}{\langle x_1, \ldots, x_n \rangle}
\newcommand{\cofilter}[1]{\Delta (#1)}
\newcommand{\filter}[1]{#1{\uparrow}}
\newcommand{\ohne}{\setminus}
\newcommand{\union}{\,\cup\,}
 \renewcommand{\emptyset}{\varnothing}
\date{\today}
\usepackage[colorlinks=true,linkcolor=blue,filecolor=blue,citecolor=blue,urlcolor=blue]{hyperref}				%

\begin{document}
\bibliographystyle{is-alpha}

\title{Structured and Punctured Nullstellens\"atze}
\author{Erhard Aichinger}
\address{Erhard Aichinger,
Institute for Algebra,
Johannes Kepler University Linz,
4040 Linz,
Austria}
\email{\tt erhard@algebra.uni-linz.ac.at}

\author{John R.\ Schmitt}
\address{John R.\ Schmitt,
  Department of Mathematics, Middlebury College, Middlebury, VT 05753, United States of America}
\email{\tt jschmitt@middlebury.edu}

\author{Henry Zhan}
\address{Henry Zhan,
  Department of Mathematics, Middlebury College, Middlebury, VT 05753, United States of America}
\email{\tt hengz@middlebury.edu}

\begin{abstract}
  A \emph{Nullstellensatz} is a theorem providing information
on polynomials that vanish on a certain set: David Hilbert's Nullstellensatz
(1893) is a cornerstone of algebraic geometry, and
Noga Alon's Combinatorial Nullstellensatz (1999) is
a powerful tool in the \emph{Polynomial Method}, a technique
used in combinatorics. Alon's Theorem excludes that a polynomial
vanishing on a grid contains a monomial with certain
properties. This theorem has been generalized in several directions,
two of which we will consider in detail:
Terence Tao and Van H.\ Vu (2006), Uwe Schauz (2008) and   
Micha\l{} Laso\'n (2010) exclude more monomials, and recently, Bogdan Nica (2023)
  improved the result for grids with additional symmetries
  in their side edges. Simeon Ball and Oriol Serra (2009) incorporated
  the multiplicity of zeros and gave Nullstellens\"atze
  for \emph{punctured grids}, which are sets of
  the form $X \setminus Y$ with both $X,Y$ grids.

  We generalize some of these results; in particular,
  we provide a common generalization to the results of Schauz and
  Nica. To this end, we establish that during multivariate polynomial
  division, certain monomials are unaffected.
  This also allows us to generalize Pete L.\ Clark's proof of the nonzero counting
  theorem by Alon and F\"uredi to punctured grids.
\end{abstract}

  \maketitle

\section{Introduction} \label{sec:intro}
For a field $\ob{K}$, $n \in \N$ and a subset $S$ of
$\kn$, we say that a polynomial $f \in \kxn$ \emph{vanishes on $S$}
if $f(\vb{a}) = 0$ for all $\vb{a} \in S$. We will be particularly
interested in the case that $S$ is a grid.
Here we say that a
subset $S$ of $\kn$ is a \emph{grid over $\ob{K}$} if there are finite subsets
$S_1, \ldots, S_n$ of $\ob{K}$ such that
$S = \bigtimes_{i=1}^n S_i$. For a polynomial $f \in \kxn$, we denote its total degree by $\deg (f)$ (with $\deg (0) = -\infty$), and
$\ul{n} := \{1,2,\ldots, n\}$.
The model of our results will be Alon's Combinatorial Nullstellensatz.
\begin{thm}[Alon's Combinatorial Nullstellensatz
  {\cite[Theorem~1.2]{Al:CN}}] \label{thm:alon} 
  Let $S = \bigtimes_{i=1}^n S_i$ be a grid over $\kk$,
  and let $f \in \kxn$ be such that $f$ contains a monomial
  $x_1^{\alpha_1} \cdots x_n^{\alpha_n}$ with
  $\alpha_i < |S_i|$ for all $i \in \ul{n}$.
  Then if
  \begin{equation} \label{eq:alon}
    \sum_{i=1}^n \alpha_i = \deg (f),
  \end{equation}
  there is $\vb{s} \in S$ such that $f(\vb{s}) \neq 0$.
\end{thm}
The proof relies on the fact (see \cite[Theorem~1.1]{Al:CN})
that the set $\I (S)$ of those polynomials
in $\kxn$ that vanish on $S$ is the ideal of $\kxn$ generated
by $g_1, \ldots, g_n$, where $g_i := \prod_{a \in S_i} (x_i-a)$, and
on the fact that every $f \in \I(S)$ can be written as
$\sum_{i=1}^n h_i g_i$ with $\deg (h_i g_i)\le \deg (f)$ for all $i \in \ul{n}$.
For ensuring the existence of a nonzero on a grid, Alon's Theorem requires that
$f$ contains a monomial
of maximal total degree such that the degree in each variable
is smaller than the corresponding side length of the grid.
Several subsequent results relax the condition~\eqref{eq:alon}
on such a monomial, and a simple condition was given in \cite{La:AGOC}:
For $\gb{\alpha} = (\alpha_1, \ldots, \alpha_n)$
and $\gb{\beta} = (\beta_1, \ldots, \beta_n) \in \NON$, we write
$\gb{\alpha} \sqsubseteq \gb{\beta}$ if $\alpha_i \le \beta_i$ for
all $i \in \ul{n}$, 
and $\gb{\alpha} \sqsubset \gb{\beta}$ if
 $\gb{\alpha} \sqsubseteq \gb{\beta}$ and $\gb{\alpha} \neq \gb{\beta}$.
The monomial
$x_1^{\alpha_1} \cdots x_n^{\alpha_n}$ is also written as $\mono{\alpha}$.
Clearly, a monomial $\mono{\alpha}$
divides a monomial $\mono{\beta}$ if and only if $\gb{\alpha} \sqsubseteq \gb{\beta}$.
For a polynomial $f = \sum_{\gb{\alpha} \in \NON} c_{\gb{\alpha}} \mono{\alpha}$,
we let $\Mon (f) := \{ \mono{\alpha} \mid \alpha \in \NON, c_{\gb{\alpha}} \neq 0 \}$ be the
set of monomials that appear in $f$, and
$\Supp (f) = \{ \alpha \in \NON \mid c_{\gb{\alpha}} \neq 0 \}$ be the
set of exponents of these monomials, called the \emph{support}
of $f$.
Now \cite[Theorem~2]{La:AGOC} tells that Theorem~\ref{thm:alon} still holds if
we replace~\eqref{eq:alon} by the weaker condition
\begin{equation} \label{eq:lason}
  (\alpha_1, \ldots, \alpha_n) \text{ is maximal in }
  \Supp (f)
  \text{ with respect to } \sqsubseteq.
\end{equation}  
Stated differently,~\eqref{eq:lason} requires that $\Supp(f)$ does not
contain a $\gamma \in \NON$ with $\alpha \sqsubseteq \gamma$ and
$\alpha \neq \gamma$. For fields of characteristic~$0$, this
result had been contained in \cite[Exercise~9.1.4]{TV:AC}.
This condition~\eqref{eq:lason} can also be stated as
\begin{equation} \label{eq:lason2}
  \text{ for every monomial }\mono{\gamma} \in \Mon (f) \setminus \{\mono{\alpha}\},
  \text{ there is } i \in \ul{n} \text{ such that } \gamma_i < \alpha_i.
\end{equation}
A stronger result is given in \cite[Theorem~3.2(ii)]{Sc:ASPD}.
This results tells that Theorem~\ref{thm:alon} still holds if
we replace~\eqref{eq:alon} by 
\begin{multline} \label{eq:schauz}
    \text{ for every monomial }\mono{\gamma} \in \Mon (f) \setminus \{\mono{\alpha}\}, \\
    \text{ there is } i \in \ul{n} \text{ such that } \gamma_i \neq \alpha_i
    \text{ and } \gamma_i \le |S_i| - 1. 
\end{multline}
Schauz's result also applies to rings other than fields. In the present
note, we restrict our attention to grids over fields.
In \cite{BS:PCN},
Ball and Serra incorporate the multiplicity
of zeros into Alon's theorem, and they extend the result from
grids to \emph{punctured grids};
these are sets that can be written as $X \setminus Y$ with both
$X$ and $Y$ grids. K\'os and R\'onyai \cite{KR:ANFM} generalized
Alon's theorem to grids whose edges are \emph{multisets}; such
grids will be considered in Section~\ref{sec:multisets}.

Nica \cite[Theorem~3.1]{Ni:POSG} gives a different lever to achieve
generalizations of Theorem~\ref{thm:alon} by taking into account the
structure
of the grid. For $\lambda \in \N_0$, 
we call a univariate polynomial $f \in \ob{K}[x]$ of degree $\nu \in \N_0$
\emph{$\lambda$-lacunary} if in $f$, all coefficients of $x^{\alpha}$ with
$\nu - \lambda \le \alpha < \nu$ vanish. Then \cite{Ni:POSG} defines
a finite set $A \subseteq \ob{K}$ to be \emph{$\lambda$-null}
if the polynomial $\prod_{a \in A} (x-a)$ is $\lambda$-lacunary.
For example, over the complex numbers $\C$, the set
$\{a \in \C \mid a^n = 1\}$ is $(n-1)$-null because the polynomial
$x^n - 1$ is $(n-1)$-lacunary, every finite subset of a field
$\ob{K}$ is $0$-null, and a subset $S$ of $\ob{K}$ is $1$-null if
$\sum_{a \in S} a = 0$. 
Nica's Theorem states:
\begin{thm}[Nica's Combinatorial Nullstellensatz for
            Structured Grids, {\cite[Theorem~3.1]{Ni:POSG}}] \label{thm:nica} 
    Let $S = \bigtimes_{i=1}^n S_i$ be a grid over the field $\ob{K}$, and
    let $\lambda \in \N_0$ be such that each $S_i$ is $\lambda$-null.
    Let $f \in \kxn$ 
    and let $(\alpha_1, \ldots, \alpha_n) \in \NON$
    with $\alpha_i < |S_i|$ for all $i \in \ul{n}$ be such that
    $f$ contains the monomial
    $x_1^{\alpha_1} \cdots x_n^{\alpha_n}$.
    Then if
  \begin{equation} \label{eq:nica}
    \sum_{i=1}^n \alpha_i \ge  \deg (f) - \lambda
  \end{equation}
  there is $\vb{s} \in S$ such that $f(\vb{s}) \neq 0$.
  \end{thm}
We can therefore see that Alon's theorem has already been extended along
four lines in the literature: In one line are the extensions of Ball and Serra
\cite{BS:PCN} and K\'{o}s and R\'{o}nyai \cite{KR:ANFM} including the \emph{multiplicities of zeros}.
Another line is the extension of the theorems to sets that are not grids as in Ball and Serra's extension to \emph{punctured grids}.
Nica's extension applies to \emph{structured grids},
and Laso\'n's, Tao and Vu's and Schauz's extensions
put different conditions on the set of \emph{monomials} appearing
in the polynomial.

In the present paper, we combine these threads and obtain generalizations
of some of these theorems. Essential in our proofs is an analysis of multivariate polynomial division; here we borrow some terms from the theory of Gr\"obner bases \cite{Bu:GBAA,BW:GB,AL:AITG}.

\section{Results}
Our first result incorporates Nica's improvement of the Combinatorial
Nullstellensatz for structured grids (\cite[Theorem~3.1]{Ni:POSG})
into Schauz's result \cite[Theorem~3.2(ii)]{Sc:ASPD} and surprisingly
yields more than the union of the two statements.
For $a, b \in \N_0$, the interval $[a, b]$ is defined
by $[a, b] := \{ x \in \N_0 \mid a \le x \le b\}$. For $a > b$, we then
have $[a,b] = \emptyset$.
  \begin{thm}[Structured Nullstellensatz using conditions on the monomials] \label{thm:cnv}
    Let $n \in \N$ and let $\lambda_1, \ldots, \lambda_n \in \N_0$.
  For each $i \in \ul{n}$, let $S_i$ be a $\lambda_i$-null
  subset of the field $\ob{K}$, and let $S := \bigtimes_{i=1}^n S_i$.
  Let $f \in \kxn$
  and let $(\alpha_1, \ldots, \alpha_n) \in \NON$
  with $\alpha_i < |S_i|$ for all $i \in \ul{n}$
  be such that $f$ contains the monomial
  $x_1^{\alpha_1} \cdots x_n^{\alpha_n}$. Furthermore,
  we assume that for every
  monomial $\mono{\gamma}$ in $\Mon (f) \setminus \{\mono{\alpha}\}$,
  there is $i \in \ul{n}$ such that
    \begin{equation} \label{eq:aich3}
      \gamma_i \in [0, \alpha_i - 1] \union [\alpha_i + 1, |S_i| - 1]
               \union [|S_i|, \alpha_i + \lambda_i].
  \end{equation}
  Then there is $\vb{s} \in S$ such that
  $f(\vb{s}) \neq 0$.
  \end{thm}
  The proof is given in Section~\ref{sec:cnv-proof}.
  We note that if we replace~\eqref{eq:aich3}
  by
  \[
  \gamma_i \in [0, \alpha_i - 1],
  \]
  we obtain Laso\'n's result \cite[Theorem~2]{La:AGOC}, 
  if we replace~\eqref{eq:aich3} by
  \[
  \gamma_i \in [0, \alpha_i - 1] \union [\alpha_i + 1, |S_i| - 1],
  \]
  we obtain Schauz's result \cite[Theorem~3.2(ii)]{Sc:ASPD}, and if 
  we replace~\eqref{eq:aich3} by
  \[
  \gamma_i \in [0, \alpha_i - 1] \union [\alpha_i + 1, \alpha_i + \lambda_i],
  \]
  we obtain a result that implies Theorem~\ref{thm:nica}.
  For this purpose, we note that
  \begin{equation} \label{eq:alS}
     [\alpha_i + 1, \alpha_i + \lambda_i] \subseteq
     [\alpha_i + 1, |S_i| - 1] \union [|S_i|, \alpha_i + \lambda_i].
  \end{equation}   
  If $|S_i| - 1 > \alpha_i + \lambda_i$, then the inclusion
  is proper.
  Hence Theorem~\ref{thm:cnv} generalizes these three results.
  An extension to multisets is given in Theorem~\ref{thm:multisetscnii}.
  
  Let us compare Theorem~\ref{thm:cnv} to other Nullstellens\"atze by
  looking at an example: Consider the $4$-null sets
$S_1 = S_2 = \{ z \in \ob{C} \mid z^5 = 1 \}$,
  let $S := S_1 \times S_2$, and let $\mono{\alpha} := x_1^2 x_2^3$.
  Each of the compared results yields a set of
  monomials $M$ such that 
    every polynomial that is a sum
of $x_1^2 x_2^3$ and a linear combination of monomials in $M$
has a nonzero in $S$. In 
Figure~\ref{fig:M} (made with
Mathematica~\cite{WR:MV14}) we draw the representations
$\{ (\gamma_1, \gamma_2) \in \N_0^2 \mid x_1^{\gamma_1} x_2^{\gamma_2} \in M\}$
of these sets of monomials.
\begin{figure} 
     \begin{subfigure}[c]{0.3 \textwidth}
     \includegraphics[width=\textwidth]{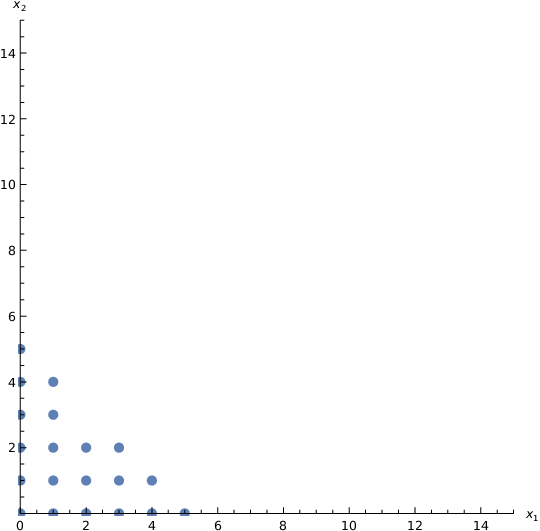}
     \subcaption*{    $M_{\text{A}}$, \\
       \cite[Theorem~1.2]{Al:CN}}
     \end{subfigure} 
     \begin{subfigure}[c]{0.3 \textwidth}
     \includegraphics[width=\textwidth]{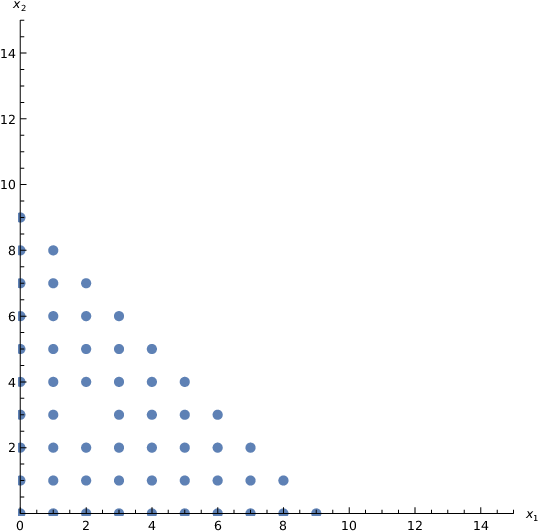}
     \subcaption*{    $M_{\text{N}}$,\\
             {\cite[Theorem~3.1]{Ni:POSG}}}
     \end{subfigure} 
     \begin{subfigure}[c]{0.3 \textwidth}
     \includegraphics[width=\textwidth]{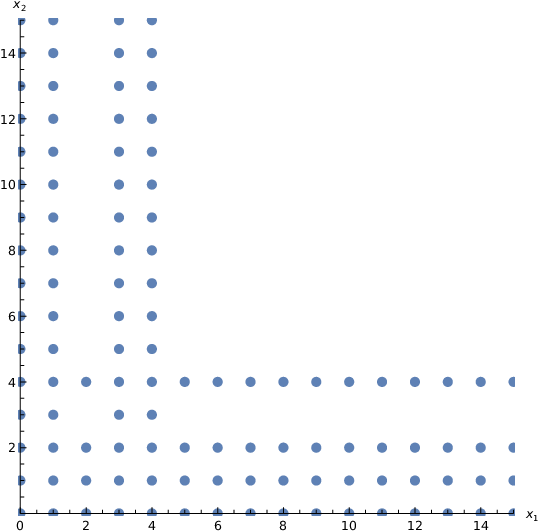}
     \subcaption*{    $M_{\text{S}}$, \\
             \cite[Theorem~3.2(ii)]{Sc:ASPD}}
     \end{subfigure} 
     \begin{subfigure}[c]{0.3 \textwidth}
     \includegraphics[width=\textwidth]{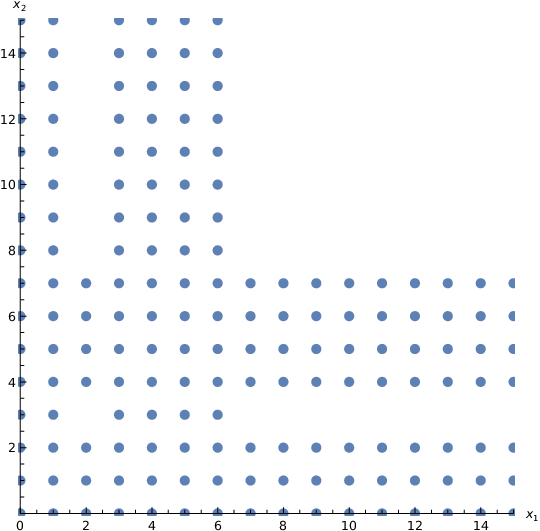}
     \subcaption*{    $M_\text{ASZ}$, \\
              Theorem~\ref{thm:cnv}}
     \end{subfigure}
     \caption{$x_1^2 x_2^3 \,\, +$ any linear combinations of the dotted
       monomials does not vanish on $S = \{(x_1, x_2) \in \C^2 \mid
       x_1^5 = x_2^5 = 1 \}$.}
  \label{fig:M}   
\end{figure}
For an in-depth comparison of allowable monomials for various
versions of the Nullstellensatz, we point the reader to
\cite{Ro:TGCL}.
  
  Our next goal is to incorporate multiplicities.
  Let $\ob{K}$ be a field, let $f \in \kxn$, let
$t \in \N_0$, and let $\vb{c} = (c_1, \ldots, c_n) \in \ob{K}$.
We say that $\vb{c}$ is a \emph{zero of multiplicity $t$}
or a \emph{$t$-fold zero} of $f$ if
the polynomial
$f' := f (c_1 + x_1, \ldots, c_n + x_n)$
lies in the ideal
$\langle x_1, \ldots, x_n \rangle^t$, which holds if and only if
$f'$ contains no monomials
of total degree less than $t$.
We note that every $\vb{c} \in \kn$ is a $0$-fold zero of $f$,
and that $\vb{c}$ is a $1$-fold zero of $f$ if and only if
$f(\vb{c}) = 0$. Furthermore, in our definition, a $t$-fold zero
is a $t'$-fold zero for all $t' \le t$.
\begin{thm}[Structured Nullstellensatz using conditions on the monomials
            with multiplicities] \label{thm:cnv-mult}
    Let $n, t \in \N$ and let $\lambda \in \NON$.
  For each $i \in \ul{n}$, let $S_i$ be a $\lambda_i$-null
  subset of the field $\ob{K}$, and let $S := \bigtimes_{i=1}^n S_i$.
  Let $f \in \kxn$
  and let $(\alpha_1, \ldots, \alpha_n) \in \NON$ be
  such that
  \begin{equation} \label{eq:betat}
    \text{for all }\beta \in \NON : \sum_{i=1}^n \beta_i = t
    \Longrightarrow (\exists i \in \ul{n} : \alpha_i < \beta_i |S_i|)
  \end{equation}
  and
  $f$ contains the monomial
  $x_1^{\alpha_1} \cdots x_n^{\alpha_n}$. Furthermore, we assume that
  for every monomial $\mono{\gamma}$ in $\Mon (f) \setminus \{\mono{\alpha}\}$,
  there is $i \in \ul{n}$ with $\gamma_i \neq \alpha_i$  such that
  \begin{multline} \label{eq:aich3mult}
    \gamma_i \in [0, \alpha_i - 1] \union
                 [\alpha_{i} + 1, \alpha_i + \lambda_i] \text{ or } \\
      \forall \beta \in \NON :\big( (\sum_{j=1}^n \beta_j = t \text{ and } \beta_i > 0)
               \Rightarrow (\exists j \in \ul{n} : \gamma_j < \beta_j |S_j| ) \big).
    \end{multline}
  Then there is $\vb{s} \in S$ such that
  $\vb{s}$ is not a $t$-fold zero of $f$.
\end{thm}
The proof is given in Section~\ref{sec:cnv-mult}.
Setting $t=1$, we obtain Theorem~\ref{thm:cnv} since
$\vb{s}$ is a $1$-fold zero of $f$ if and
only if $f(\vb{s}) \neq 0$. As a corollary, we obtain a
common generalization of \cite[Theorem~3.1]{Ni:POSG}
and \cite[Corollary~3.2]{BS:PCN}.
\begin{cor}[Structured Nullstellensatz 
            with multiplicities] \label{cor:nica-mult}
    Let $n \in \N$ and let $\lambda \in \N_0$.
  For each $i \in \ul{n}$, let $S_i$ be a $\lambda$-null
  subset of the field $\ob{K}$, and let $S := \bigtimes_{i=1}^n S_i$.
  Let $f \in \kxn$
  and let $(\alpha_1, \ldots, \alpha_n) \in \NON$ be
  such that
  \begin{equation*}
    \text{for all }\beta \in \NON : \sum_{i=1}^n \beta_i = t
    \Longrightarrow (\exists i \in \ul{n} : \alpha_i < \beta_i |S_i|),
  \end{equation*}
  $f$ contains the monomial
  $x_1^{\alpha_1} \cdots x_n^{\alpha_n}$, and 
  \begin{equation} \label{eq:nica2}
     \sum_{i=1}^n \alpha_i \ge \deg (f) - \lambda.
  \end{equation}
  Then there is $\vb{s} \in S$ such that
  $\vb{s}$ is not a $t$-fold zero of $f$.
\end{cor}

Next, we consider generalizations from grids to \emph{punctured grids}.
\begin{thm}[Structured Nullstellensatz for punctured grids using
            conditions on the monomials] \label{thm:ballcn2}
  Let $X = \bigtimes_{i=1}^n X_i, Y = \bigtimes_{i=1}^n Y_i$ be grids over the field $\kk$ with $Y_i \subseteq X_i$ for all $i \in \ul{n}$, let
  $P := X \setminus Y$, and let
  $\lambda_1, \ldots, \lambda_n \in \N_0$. We assume that
  for each $i \in \ul{n}$, both $X_i$ and $Y_i$ are $\lambda_i$-null.
    Let $f \in \kxn$
  and let $(\alpha_1, \ldots, \alpha_n) \in \Supp (f)$ be such that
  \begin{enumerate}
  \item \label{it:al1} for all $i \in \ul{n}$, $\alpha_i < |X_i|$,
  \item \label{it:al2} there exists $i \in \ul{n}$ such that $\alpha_i <  |X_i| - |Y_i|$,
  \item \label{it:al3} for all $\mono{\gamma} \in \Mon (f)$, there exists
    $i \in \ul{n}$ such that at least one of the following three conditions
    holds:
    \begin{enumerate}
    \item \label{it:al3a}
      $\gamma_i \in [0, \alpha_i - 1] \union [\alpha_i + 1, \alpha_i + \lambda_i]$,
    \item \label{it:al3b}
      $\gamma_i \in [\alpha_i + 1, |X_i| - 1]$ and $|X_i| = |Y_i|$,
    \item \label{it:al3c}  
      $\gamma_i \in [\alpha_i + 1, |X_i| - 1]$ and there is
        $j \in \ul{n}$ with $\gamma_j < |X_j| - |Y_j|$.
    \end{enumerate}
 \end{enumerate}   
  Then there is $\vb{z} \in P$ with $f (\vb{z}) \neq 0$.
\end{thm}
The proof is given in Section~\ref{sec:punct}.
As a consequence, we obtain:
\begin{cor}[Structured Nullstellensatz for punctured grids] \label{cor:ballcn2cor}
Let $X = \bigtimes_{i=1}^n X_i, Y = \bigtimes_{i=1}^n Y_i$ be grids over the field $\kk$ with $Y_i \subseteq X_i$ for all $i \in \ul{n}$, let
  $P := X \setminus Y$, and let
  $\lambda \in \N_0$. We assume that
  for each $i \in \ul{n}$, both $X_i$ and $Y_i$ are $\lambda$-null.
    Let $f \in \kxn$
  and let $(\alpha_1, \ldots, \alpha_n) \in \Supp (f)$ be such that
  \begin{enumerate}
  \item \label{it:al1cor} for all $i \in \ul{n}$, $\alpha_i < |X_i|$,
  \item \label{it:al2cor} there exists $i \in \ul{n}$ such that $\alpha_i <  |X_i| - |Y_i|$,
  \item \label{it:al3new}
         $\sum_{i=1}^n \alpha_i \ge \deg (f) - \lambda$.
  \end{enumerate}
  Then there is $\vb{z} \in P$ with $f (\vb{z}) \neq 0$.
\end{cor}
The investigation of punctured grids also yields the following extension
of the Alon-F\"uredi Nonzero Counting Theorem
\cite[Theorem~5]{AF:CTCB}. The recent book manuscript
by Clark \cite{Cl:ATCW} contains other extensions and was the basis
for ours. For $f \in \kxn$, we write
$\V(f)$ for the set $\{\vb{a} \in \kn \mid f(\vb{a}) = 0\}$
of zeros of $f$.
\begin{thm}[Nonzero counting for punctured grids] \label{thm:AFC-punct}
  Let $X = \bigtimes_{i=1}^n X_i$ and
      $Y = \bigtimes_{i=1}^n Y_i$ be grids over the field $\kk$, let
  $P := X \setminus Y$, and let
  $f \in \kxn \setminus \{0\}$. For $i \in \ul{n}$, let
  $a_i := |X_i|$ and $b_i := |Y_i|$. %
  \begin{enumerate}
    \item \label{it:bin1}
      Let
    \begin{multline*} A := \{ (y_1, \ldots, y_n) \in \N^n \mid \\
    \forall i \in \ul{n} : 1 \le y_i \le a_i, \,\,  %
    \exists i \in \ul{n} : y_i > b_i, \text{ and }
    \sum_{i=1}^n y_i \ge \sum_{i=1}^n a_i - \deg (f) \}.
    \end{multline*}
    If $P \ohne \V(f) \neq \emptyset$, then 
    \begin{equation} \label{eq:Zf}
       |P \ohne V(f)| \ge \min \{ \prod_{i=1}^n y_i - \prod_{i=1}^n \min (y_i, b_i)
       \mid (y_1, \ldots, y_n) \in A \}.
    \end{equation} 
  \item \label{it:bin2}
    We assume that for all $i \in \ul{n}$,
    we have $\deg_{x_i} (f) < a_i$. 
    Let
    \[
       B := \{(y_1, \ldots, y_n) \in \N^n \mid
         \forall i \in \ul{n} : a_i - \deg_{x_i} (f) \le y_i \le a_i, \text{ and }
    \sum_{i=1}^n y_i = \sum_{i=1}^n a_i - \deg (f) \}.
    \]
    If $P \ohne \V(f) \neq \emptyset$, then  
    \begin{equation} \label{eq:Zf2} 
       |P \ohne V(f)| \ge \min \{ \prod_{i=1}^n y_i - \prod_{i=1}^n \min (y_i, b_i)
       \mid (y_1, \ldots, y_n) \in B \}.
    \end{equation}   
  \end{enumerate}
\end{thm}
The proof is given in Section~\ref{sec:AF}.
Let us give an overview how these results are proved.
Theorem~\ref{thm:cnv} claims that a polynomial containing
certain monomials does not vanish on the whole grid $S$.
From \cite[Theorem~1.1]{Al:CN} we know that the ideal $\I (S)$ of polynomials
vanishing on $S$ is generated by
$\{ g_i \mid i \in \ul{n} \}$ with
$g_i := \prod_{a \in S_i} (x_i - a)$. The polynomial $g_i$ has leading
term $x_i^{|S_i|}$. For proving $f \not\in \I(S)$, we show that its
remainder $r$ modulo $G = \{ g_i \mid i \in \ul{n} \}$ after multivariate
polynomial division by $G$ is nonzero. The conditions on $f$ and the
$g_i$'s ensure that the monomial $\mono{\alpha}$ can never
be reduced in the course of the division, and that all other monomials
in $f$ have too small exponents to be able to produce a term
$c_{\alpha} \mono{\alpha}$ that would allow to cancel $\mono{\alpha}$
before it stays -- remains -- in the remainder.
Such an approach only works if all polynomials in $\I(S)$ will
have remainder $0$ after multivariate division by $G$. This is guaranteed
when $G$ is not only a generating set, but furthermore even
a \emph{Gr\"obner basis} of $\I(S)$;
for the generating set $G$ considered in the proof of
Theorem~\ref{thm:cnv} this is ensured by the fact that leading monomials
of the polynomials in $G$ are coprime to each other.
In order to state these ideas
precisely, we will make use of some notions from the arithmetic
of multivariate polynomials, in particular of multivariate polynomial
division, which is one of the basics
of the theory of Gr\"obner bases.

\section{Lacunary multivariate polynomials} \label{sec:lacunary2}
In this section, we extend the definition of lacunary polynomials
to multivariate polynomials.

We sort monomials using \emph{admissible orderings:} A linear
order $\lea$ on $\NON$ is \emph{admissible} if it is total,
and for all $\alpha, \beta, \gamma \in \NON$, we have
$(\alpha \sqsubseteq \beta \Rightarrow \alpha \le_a \beta)$
and $(\alpha \le_a \beta \Rightarrow \alpha + \gamma \le_a \beta + \gamma)$.
When $\alpha \lea \beta$, we will also write $\mono{\alpha} \lea \mono{\beta}$,
and $\alpha <_a \beta$ stands for ($\alpha \lea \beta$ and $\alpha \neq
\beta$).
If $\alpha \in \NON$ is maximal in
$\Supp (f)$ with respect to $\le_a$,
then $\mono{\alpha}$ is called
the \emph{leading monomial of $f$} and abbreviated by $\LM(f)$,
and $\alpha$ is the \emph{leading exponent} or \emph{multidegree} of $f$,
abbreviated by  $\LEXP (f)$ and $\md (f)$.
The coefficient $c_{\alpha}$ of the leading monomial $\mono{\alpha}$ is
the \emph{leading coefficient}, abbreviated as $\LC(f)$,
and $c_{\alpha} \mono{\alpha} = \LC (f) \cdot \LM (f) =
    \LC(f) \mono{\LEXP (f)}$
is the \emph{leading term}
of $f$, denoted by $\LT (f)$.
Every admissible ordering is a well ordering, i.e., it is total and
has no infinite descending chains (cf. \cite[Theorem~5.5(ii)]{BW:GB};
a proof can also be found, e.g., in
the survey \cite{Ai:SGBA} (Lemma~9.2)).

\begin{de} Let $\lambda \in \NON$, and
  let $g \in \kxn$. The polynomial $g$ is
  \emph{$\lambda$-lacunary} if
  it contains a monomial $\mono{\mu}$ such that
  for each $\mono{\nu} \in \Mon(g)$ and for
  each $i \in \ul{n}$, we have
  $\nu_i < \mu_i - \lambda_i$ or $\nu_i = \mu_i$.
  A set $G \subseteq \kxn$ is called $\lambda$-lacunary if
  every $g \in G$ is $\lambda$-lacunary.
\end{de}
We note that then $\nu \sqsubseteq \mu$ for all
$\mono{\nu} \in \Mon (g)$, and therefore $\mono{\mu}$ is the leading
monomial of $g$ with respect to every admissible monomial ordering
$\lea$.
\begin{lem} \label{lem:proquolac}
  Let $\lambda \in \NON$, and let $f,g,h \in \kxn$ with
  $f = gh$.
  Then we have:
  \begin{enumerate}
  \item \label{it:pq1} If $g$ and $h$ are $\lambda$-lacunary, then
                     $f$ is $\lambda$-lacunary.
  \item \label{it:pq2} If $f$ and $g$ are $\lambda$-lacunary, then
    $h$ is $\lambda$-lacunary.
  \end{enumerate}
\end{lem}
\begin{proof}
  Let $\mono{\mu} := \LM (f)$, $\mono{\nu} := \LM (g)$ and
  $\mono{\rho} := \LM (h)$. Then $\mu = \nu + \rho$.
  For showing~\eqref{it:pq1}, we 
  fix $i \in \ul{n}$ and show that for each
   monomial $\mono{\alpha} \in \Mon (f)$, we
   have $\alpha_i = \mu_i$ or $\alpha_i < \mu_i - \lambda_i$.
   We write $f$ as $\sum_{j=0}^{\mu_i} f_j x_i^j$, where
   $f_j \in \kk [x_1,\ldots, x_{i-1},x_{i+1}, \ldots, x_n]$
   for all $j$ with $0 \le j \le \mu_i$, and we set $f_j = 0$
   for $j > \mu_i$. Similarly,
   \begin{equation} \label{eq:ghi}
   g = \sum_{j=0}^{\nu_i} g_j x_i^j \text{ and }
   h = \sum_{j=0}^{\rho_i} h_j x_i^j.
   \end{equation}
   Since $g$ and $h$ are $\lambda$-lacunary,
   we have $g_j = 0$ for all $j$ with
   $\nu_i - \lambda_i \le j \le \nu_i - 1$ and
   $h_j = 0$ for all $j$ with
   $\rho_i - \lambda_i \le j \le \rho_i - 1$.
   Now let $k \in \N$ be such that
   \[
      \mu_i - \lambda_i \le k \le \mu_i - 1.
   \]
   Then $f_k = \sum_{l = 0}^{k} g_l h_{k-l}$. We will show
   $f_k = 0$ by establishing that all
   $k+1$ summands are $0$. To this end, we let $l \in \{0,\ldots,k\}$.
   If $l < \nu_i - \lambda_i$, then
   $k-l > k - \nu_i + \lambda_i \ge \mu_i - \lambda_i - \nu_i + \lambda_i
   = \rho_i$, and thus $h_{k-l} = 0$.
   If $\nu_i - \lambda_i \le l \le \nu_i - 1$,
   we have $g_l = 0$. If $\nu_i \le l \le k - \rho_i + \lambda_i$,
   we have $k-l \ge \rho_i - \lambda_i$ and
   $k-l \le k - \nu_i \le \mu_i - 1 - \nu_i = \rho_i - 1$
   and thus $h_{k-l} = 0$.
   If $l \ge k - \rho_i +  \lambda_i + 1$, then
   $l \ge \mu_i - \lambda_i - \rho_i + \lambda_i + 1 =
   \nu_i + 1$, and therefore $g_l = 0$.
   Thus $f_k = 0$.
   Hence $f$ contains no monomial $\mono{\alpha}$ with
   $\mu_i - \lambda_i \le \alpha_i \le \mu_i - 1$.
   This completes the proof of~\eqref{it:pq1}.

   For proving~\eqref{it:pq2}, we assume that $g$ is
   $\lambda$-lacunary and $h$ is not $\lambda$-lacunary.
   Then there are $\mono{\alpha} \in \Mon (h)$ and
   $i \in \ul{n}$ with $$\rho_i - \lambda_i \le \alpha_i \le \rho_i - 1.$$
   Again, we write $f = \sum_{j=0}^{\mu_i} f_j x_i^j$,
   $g = \sum_{j=0}^{\nu_i} g_j x_i^j$ and 
   $h = \sum_{j=0}^{\rho_i} h_j x_i^j$.
   Since $\mono{\alpha} \in \Mon(h)$, we have    
   $h_{\alpha_i} \neq 0$. 
    We have 
   \[
   f_{\nu_i + \alpha_i} = \sum_{l = 0}^{\nu_i + \alpha_i} g_{\nu_i + \alpha_i - l} h_l.
   \]
   For $l < \alpha_i$, we have $g_{\nu_i + \alpha_i - l} = 0$. For $l = \alpha_i$, we
   obtain the summand $g_{\nu_i} h_{\alpha_i}$. For $l$ with $\alpha_i < l \le \alpha_i + \lambda_i$,
   we obtain $\nu_i - \lambda_i \le \nu_i + \alpha_i - l < \nu_i$, and therefore
   $g_{\nu_i + \alpha_i - l} h_l = 0$. For $l > \alpha_i + \lambda_i$, we have   
   $l > \rho_i$, and therefore $h_l = 0$.
   Hence
   $f_{\nu_i + \alpha_i} = g_{\nu_i} h_{\alpha_i} \neq 0$. Since
   $\mu_i - \lambda_i  = \nu_i + \rho_i - \lambda_i
   \le \nu_i + \alpha_i \le \nu_i + \rho_i - 1 = \mu_i - 1$,
   $f$ can then not be $\lambda$-lacunary.
\end{proof}

\section{Multivariate Polynomial Division} \label{sec:division}
In this section, we analyze the stability of certain monomials
during multivariate polynomial division. Over the integers,
a division of $f$ by $g$ with $g \neq 0$ produces a quotient
$h$ and a remainder $r$ with $f = hg+r$ and $|r| < |g|$. If
$f, g_1,\ldots, g_s$ are multivariate polynomials in $\kxn$,
then division produces an expression
$f = \sum_{i=1}^s h_i g_i +r$ with certain properties of both
the ``quotients'' $h_1, \ldots, h_s$ and the remainder $r$.
Following \cite{Ei:CA}, the equation $f = \sum_{i=1}^s h_i g_i +r$
is then called a \emph{standard expression}.
We will need to write the $h_i$'s as sums of monomials,
and we observe
that the multivariate polynomial division algorithm
explained, e.g., in \cite[Proposition~5.22]{BW:GB}, \cite[Chapter~2, \S 3]{CLO:IVAA4}, 
\cite[Algorithm~1.5.1]{AL:AITG} or \cite[Algorithm~2.3.4]{Sm:ITAG} can easily be modified to produce
what we will call a \emph{natural standard expression.} Every natural
standard expression with remainder $0$ is a \emph{standard representation} in
the sense of \cite[Definition~5.59]{BW:GB}.
\begin{de} \label{de:nse}
Let $G \subseteq \kxn \setminus \{0\}$.
A \emph{natural standard expression} of $f$ by $G$ 
with remainder $r$ with respect to the admissible ordering
$\lea$ is an equality
\begin{equation} \label{eq:se}
f = \sum_{j=1}^t c_j \,\mono{\delta_j} g_j + r,
\end{equation}
where $t \in \N_0$, $c_1, \ldots, c_t \in \ob{K} \setminus \{0\}$,
$\delta_{1}, \ldots,  \delta_{t} \in \NON$,
$g_1, \ldots, g_t \in G$, $r \in \kxn$,
and for each $j \in \ul{t}$, we have
\begin{equation} \label{eq:nat}
  \LM (c_j \,\mono{\delta_{j}} g_j) \in
  \Mon (f - \sum_{i=1}^{j-1} c_i \mono{\delta_i} g_i)
\end{equation}  
and
\begin{equation} \label{eq:nat2}
  \LM (c_j \,\mono{\delta_j} g_j) \not\in
  \Mon (f - \sum_{i=1}^{j} c_i \mono{\delta_i} g_i),
\end{equation}  
and $r$ does not contain a monomial that is divisible by
any monomial in $\{ \LM (g) \mid g \in G\}$.
\end{de}
This definition expresses that during the $j$-th step in the
division of $f$ by $G$, the term $c_j \mono{\delta_j} \LM(g_j)$
appears in the intermediate polynomial that we seek to reduce,
and this term is eliminated by subtracting
$c_j \mono{\delta_j} g_j$.
There are two differences to standard expressions as
used in the literature: first
(cf. \cite[p.334]{Ei:CA}), standard expressions
are often written in a collected form $\sum_{i=1}^s h_i g_i + r$ with
$h_i \in \kxn$.
The second main difference is that a standard representation
as defined in \cite[Definition~5.59]{BW:GB} 
need not come from an actual execution of the division algorithm;
for example setting $f = 2 x^2 y  + x$, $g_1 = xy$, $g_2 = x^2$,
we obtain $f = x \cdot g_1 + y \cdot g_2 + x$, which is a
standard representation, but an execution of the division algorithm
would always reduce the monomial $2x^2 y$ in one step,
yielding, e.g.,  the natural standard expression $f = 2 x g_1 + x$.
The definitions in \cite{Ei:CA, BW:GB} do not grasp this aspect of division, and
hence for our purposes, we prefer the refinement to
\emph{natural standard expressions} given in
Definition~\ref{de:nse}.

An important observation is that
during this division process, certain monomials of $f$
can never be reduced and will therefore end up in the
remainder $r$. We will always assume that the divisors
are $\lambda$-lacunary polynomials. We note that
for a $\lambda$-lacunary polynomial $g$, the leading monomial
is the same for all admissible monomial orderings $\lea$.
Hence the following definitions do not depend on
the choice of the admissible monomial ordering $\lea$ used to determine
$\LM(g)$. The first definition tries to identify monomials $\mono{\gamma}$
in a polynomial $f$ that, in the course of a multivariate polynomial
division of $f$ by $G$, have the potential
to produce a term $c \mono{\alpha}$ that might cancel $\mono{\alpha}$.
We will call these threats to $\mono{\alpha}$'s ability to remain
intact during the division process \emph{$(G, \lambda, \alpha)$-shading monomials.}
  \begin{de} \label{de:shading2}
    Let $\alpha, \gamma, \lambda \in \NON$ and let
    $G$ be a $\lambda$-lacunary subset of $\kxn$.
   The monomial $\mono{\gamma}$ is
  \emph{$(G, \lambda, \alpha)$-shading} if
  \begin{enumerate}
  \item \label{it:sha1}
        $\alpha \sqsubseteq \gamma$ and $\alpha \neq \gamma$,
  \item \label{it:sha2}
    for all $i \in \ul{n}$ with
    $\alpha_i < \gamma_i$,
    there is $g \in G$ with $\LM (g) \mid \mono{\gamma}$
    and $\deg_{x_i} (g) > 0$, and
  \item \label{it:sha3}
    for all $i \in \ul{n}$ with
    $\alpha_i < \gamma_i$, we have
    $\alpha_i + \lambda_i < \gamma_i$.
  \end{enumerate}
  \end{de}
  The next definition tries to single out monomials that will not
  be affected by division by $G$. Theorem~\ref{thm:gstable2}
  then shows that these monomials indeed remain intact.
  \begin{de} \label{de:stable2}
    Let $G \subseteq \kxn$, let $\alpha, \lambda \in \NON$,
    and let $f \in \kxn$.
  We say that
  $\mono{\alpha}$ is a \emph{$(G, \lambda)$-stable} monomial
  in $f$ if the following conditions hold:
  \begin{enumerate}
  \item \label{it:st21} $\alpha \in \Supp (f)$,
  \item \label{it:st22}
    there is no
    $g \in G$ with $\LM (g) \mid  \mono{\alpha}$, and
   \item \label{it:st23}  
     $f$ contains no $(G, \lambda, \alpha)$-shading monomial.
   \end{enumerate}  
\end{de}  

  \begin{thm} \label{thm:gstable2}
    Let $\lambda \in \NON$, and let
    $G$ be a $\lambda$-lacunary subset of $\kxn$.
        Let $f \in \kxn$, and let $g \in G$, $\delta \in \NON$ be such that
    $\LM (\mono{\delta} g) \in \Mon (f)$. Let
  $\mono{\alpha}$ be a $(G, \lambda)$-stable monomial in $f$,
  let $c \in \ob{K} \setminus \{0\}$, and let 
  \[
  h = f - c \,\mono{\delta} g.
  \]
  Then $\mono{\alpha}$ is a 
  $(G,\lambda)$-stable monomial in $h$.
\end{thm}  
  \begin{proof}
  Let $\mu := \LEXP (g)$.  
  We first show Condition~\eqref{it:st21}
  of~Definition~\ref{de:stable2}. This condition is
\begin{equation} \label{eq:ah2}
  \alpha \in \Supp (h).
\end{equation}  
Seeking a contradiction, we suppose $\alpha \not\in \Supp (h)$.
Then $\alpha \in \Supp (\mono{\delta} g)$.
Thus there is $\mono{\epsi} \in \Mon (g)$ with
$\mono{\alpha} = \mono{\delta} \mono{\epsi}$.
If $\mono{\epsi} = \LM (g)$, then $\LM(g) \mid \mono{\alpha}$,
which violates Condition~\eqref{it:st22} of Definition~\ref{de:stable2},
and hence $\mono{\alpha}$ is not $(G,\lambda)$-stable in $f$,
contradicting the assumptions.
In the case that $\mono{\epsi} \neq \LM (g)$, we show that
\begin{equation} \label{eq:deltamushades}
  \mono{\delta} \mono{\mu} \text{ is a $(G, \lambda, \alpha)$-shading
    monomial in $f$}.
\end{equation}
To this end, we first show that $\mono{\delta}\mono{\mu}$
satisfies Condition~\eqref{it:sha1} of Definition~\ref{de:shading2}.
Since $g$ is lacunary, we have $\epsi \sqsubseteq \mu$,
and therefore $\alpha = \delta + \epsi \sqsubseteq \delta + \mu$.
Since $\epsi \neq \mu$, we also have $\alpha \neq \delta + \mu$.
This completes the proof of Condition~\eqref{it:sha1} of
Definition~\ref{de:shading2}.
Next, we show Conditions~\eqref{it:sha2}~and~\eqref{it:sha3}
of Definition~\ref{de:shading2}.
To this end, we fix $i \in \ul{n}$ and assume $\alpha_i < \delta_i + \mu_i$.
Then $\delta_i + \epsi_i < \delta_i + \mu_i$, and therefore
$\epsi_i < \mu_i$. Hence $\mu_i \neq 0$, and thus $g$ witnesses
that Condition~\eqref{it:sha2} is satisfied. Since $G$ is
$\lambda$-lacunary, we obtain $\epsi_i < \mu_i - \lambda_i$,
which implies $\alpha_i + \lambda_i = \delta_i + \epsi_i + \lambda_i <
\delta_i + \mu_i$, completing the proof of Condition~\eqref{it:sha3}
and of~\eqref{eq:deltamushades}.
Since $\mono{\delta} \mono{\mu} \in \Mon (f)$, this monomial
violates Condition~\eqref{it:st23} of Definition~\ref{de:stable2}
and therefore witnesses that $\mono{\alpha}$ is not $(G,\lambda)$-stable
in~$f$,
contradicting the assumptions and completing the proof of~\eqref{eq:ah2}.

Continuing to show that $\mono{\alpha}$ is $(G, \lambda)$-stable
in $h$, we observe that Condition~\eqref{it:st22} of
Definition~\ref{de:stable2} is inherited from the assumption that
$\mono{\alpha}$ is $(G, \lambda)$-stable in $f$.
Hence we turn to Condition~\eqref{it:st23}. Seeking a contradiction,
we assume that $h$ contains a $(G,\lambda,\alpha)$-shading
monomial $\mono{\gamma}$. Since $f$ contains no
$(G, \lambda, \alpha)$-shading monomial, we know that
$\mono{\gamma} \in \Mon (\mono{\delta} g)$ and
$\mono{\gamma} \neq \mono{\delta}\mono{\mu}$.
Thus there is $\rho \in \Supp (g) \setminus \{\mu\}$ such
that
\[
  \mono{\gamma} = \mono{\delta} \mono{\rho}.
\]
Let $\tilde{\gamma} := \delta + \mu$.  
We show that then
\begin{equation} \label{eq:deltamushades2}
  \mono{\tilde{\gamma}} \text{ is a $(G, \lambda, \alpha)$-shading
    monomial in $f$}.
\end{equation}
To this end, we first show that $\mono{\tilde{\gamma}}$
satisfies Condition~\eqref{it:sha1} of Definition~\ref{de:shading2}.
Since $g$ is lacunary, we have $\rho \sqsubseteq \mu$,
and therefore
\begin{equation} \label{eq:gammadeltamu}
  \gamma = \delta + \rho \sqsubseteq \delta + \mu = \tilde{\gamma}.
\end{equation}  
Since $\mono{\gamma}$ is $(G,\lambda,\alpha)$-shading, we have
$\alpha \sqsubset \gamma$, and therefore
$\alpha \sqsubset \tilde{\gamma}$.
This completes the proof of Condition~\eqref{it:sha1} of
Definition~\ref{de:shading2}.
Next, we show Conditions~\eqref{it:sha2}~and~\eqref{it:sha3}
of Definition~\ref{de:shading2}.
To this end, we fix $i \in \ul{n}$ and assume that $\alpha_i < \tilde{\gamma}_i$.

We first consider the case that $\alpha_i < \gamma_i$. Since
$\mono{\gamma}$ is $(G,\lambda, \alpha)$-shading in $h$, we obtain
a $g' \in G$ with $\deg_{x_i} (g') > 0$ and $\LM(g') \mid \mono{\gamma}$,
and that $\alpha_i + \lambda_i < \gamma_i$. Since
$\gamma \sqsubseteq  \tilde{\gamma}$, we then have
$\LM (g') \mid \mono{\tilde{\gamma}}$ and
$\alpha_i + \lambda_i < \tilde{\gamma}_i$. Thus in this case,
\eqref{eq:deltamushades2} holds.

Now consider the case $\alpha_i = \gamma_i$. Then
$\gamma_i < \tilde{\gamma}_i$, and thus $\rho_i < \mu_i$.
We claim that then $g$ witnesses Condition~\eqref{it:sha2} of
Definition~\ref{de:shading2} needed to verify for proving
that $\mono{\tilde{\gamma}}$ is $(G, \lambda, \alpha)$-shading.
By the definition of $\tilde{\gamma}$, we obtain
$\LM (g) = \mono{\mu} \mid \mono{\delta} \mono{\mu} = \mono{\tilde{\gamma}}$.
Since $\mu_i > \rho_i$, we have $\deg_{x_i} (g) > 0$.
We still have to show Condition~\eqref{it:sha3}, which
claims that
\begin{equation} \label{eq:alg}
  \alpha_i + \lambda_i < \tilde{\gamma}_i.
\end{equation}  
Since $\rho_i < \mu_i$, the fact that $g$ is $\lambda$-lacunary yields
$\rho_i < \mu_i - \lambda_i$, and therefore
\(\alpha_i = \gamma_i = \delta_i + \rho_i < \delta_i + \mu_i - \lambda_i = \tilde{\gamma}_i - \lambda_i, \) establishing~\eqref{eq:alg} and completing
the proof of \eqref{eq:deltamushades2}.
Since $\mono{\delta} \mono{\mu} \in \Mon (f)$, this monomial
violates Condition~\eqref{it:st23} of Definition~\ref{de:stable2}
and therefore witnesses that $\mono{\alpha}$ is not $(G,\lambda)$-stable in $f$,
contradicting the assumptions and completing the proof that
$\mono{\alpha}$ is $(G, \lambda)$-stable
in $h$.
\end{proof}

\begin{cor} \label{cor:rem}
   Let $\lambda \in \NON$, let
   $G$ be a $\lambda$-lacunary subset of $\kxn$,
   and let
   $\lea$ be an admissible ordering of $\NON$.
   Let
   \begin{equation} \label{eq:se2}
      f = \sum_{j=1}^t c_j \, \mono{\delta_j} g_j + r
  \end{equation}
   be a natural standard expression of $f$ by $G$.
   Then all $(G,\lambda)$-stable elements of $\Mon(f)$
   are also elements of $\Mon(r)$.
\end{cor}
\begin{proof}
  Let $\mono{\alpha}$ be a $(G,\lambda)$-stable monomial
  in $f$. 
  By induction on $s$, we show that $\mono{\alpha}$
  is also $(G,\lambda)$-stable in 
  \[
       f - \sum_{j=1}^s c_j \, \mono{\delta_j} g_j.
  \]
  For $s = 0$, there is nothing to prove.
  Now assume that $s \in \{0, \ldots, t-1\}$.
  As inductive hypothesis, we assume that
  $\mono{\alpha}$ is $(G,\lambda)$-stable in
  $f - \sum_{j=1}^s c_j \, \mono{\delta_j} g_j$.
  Then by Theorem~\ref{thm:gstable2},
  $\mono{\alpha}$ is also $(G,\lambda)$-stable
  in
  $f - \sum_{j=1}^s c_j \, \mono{\delta_j} g_j -
  c_{s+1} \mono{\delta_{s+1}} g_{s+1}$,
  completing the induction step.
\end{proof}

\section{A Nullstellensatz for structured grids using conditions on the monomials}
\label{sec:cnv-proof}
In this Section, we will prove Theorem~\ref{thm:cnv}.
For $i \in \ul{n}$, we let
$f_i (x) := \prod_{a \in S_i} (x-a)$, and
we let $g_i := f_i (x_i)$.
If each $f_i$ is a univariate $\lambda_i$-lacunary polynomial
in $\ob{K}[x]$, 
then for each $i \in \ul{n}$,
$g_i$ is a
$(\lambda_1, \ldots, \lambda_n)$-lacunary polynomial in $\kxn$.
With this observation in mind, we can apply Corollary~\ref{cor:rem} to prove
the main result:
\begin{proof}[Proof of Theorem~\ref{thm:cnv}]
  Let $\lambda := (\lambda_1, \ldots, \lambda_n)$, and 
  for each $i \in \ul{n}$, 
  let $g_i := \prod_{a \in S_i} (x_i - a)$. Let
  $I$ be the ideal of $\kxn$ generated by
  $G = \{g_1, \ldots, g_n \}$. By \cite[Theorem~1.1]{Al:CN},
  a polynomial $f$ vanishes on $S_1 \times \cdots \times S_n$
  if and only if it lies in $I$. Now we seek to apply
  Corollary~\ref{cor:rem}. 
  For each $i \in \ul{n}$, $S_i$ is $\lambda_i$-null
  and therefore the polynomial $g_i$
  is $(\lambda_1, \ldots, \lambda_n)$-lacunary.
  We will now show that $\mono{\alpha}$ is a $(G,\lambda)$-stable monomial
  in $f$ with respect to $\lea$.
  First, we observe that for each $i \in \ul{n}$ we have
  $\alpha_i < |S_i|$ and therefore the monomial
  $\LM(g_i)=x_i^{|S_i|}$ does not divide $\mono{\alpha}$, which establishes
  Condition~\eqref{it:st22} of Definition~\ref{de:stable2}.
  Next, we show that $f$ contains no $(G, \lambda, \alpha)$-shading
  monomial.
  Let $\mono{\gamma} \in \Mon (f)$. If $\gamma = \alpha$, then
  $\mono{\gamma}$ violates Condition~\eqref{it:sha1} of
  Definition~\ref{de:shading2} and is therefore not
  $(G,\lambda,\alpha)$-shading.
  If $\gamma \neq \alpha$, the assumption yields
  an $i \in \ul{n}$ such that
  $$\gamma_i \in [0, \alpha_i - 1] \union [\alpha_i + 1, |S_i| - 1] \union
           [|S_i|, \alpha_i + \lambda_i].$$
  If $\gamma_i \in [0, \alpha_i - 1]$, then Condition~\eqref{it:sha1}
  of Definition~\ref{de:shading2} is violated, and so $\mono{\gamma}$
  is not $(G,\lambda,\alpha)$-shading.
  If $\gamma_i \in [\alpha_i + 1, |S_i| - 1]$ and Condition~\eqref{it:sha2}
  of Definition~\ref{de:shading2} is satisfied, then
  $\LM(g_i) \mid \mono{\gamma}$, and therefore
  $|S_i| \le \gamma_i$, contradicting $\gamma_i \le |S_i| - 1$. We
  conclude that also in the case $\gamma_i \in [\alpha_i + 1, |S_i| - 1]$,
  $\mono{\gamma}$  is not $(G,\lambda,\alpha)$-shading.
  If $\gamma_i \in [|S_i|, \alpha_i + \lambda_i]$, then
  we have $\gamma_i \ge |S_i| > \alpha_i$. If Condition~\eqref{it:sha3}
  of Definition~\ref{de:shading2} is satisfied, we have
  $\gamma_i > \alpha_i + \lambda_i$, contradicting
  $\gamma_i \le \alpha_i + \lambda_i$. Hence also in this
  case $\mono{\gamma}$ is not $(G,\lambda,\alpha)$-shading.
  Thus $f$ contains no $(G, \lambda, \alpha)$-shading
  monomial and therefore, $\mono{\alpha}$ is $(G, \lambda)$-stable.
  
     Let
   \(
      f = \sum_{j=1}^t c_j \, \mono{\delta_j} g_{i_j} + r
  \)
   be a natural standard expression of $f$ by $G$.
   Since $\mono{\alpha}$ is a $(G, \lambda)$-stable monomial in $f$,
   Corollary~\ref{cor:rem} yields that 
    $\mono{\alpha} \in \Mon(r)$. 
   Since the leading monomials of the polynomials in $G$ are coprime,
   \cite[Lemma~5.66]{BW:GB} yields that
   the set $G$ is a Gr\"obner basis of $I$
   (cf. \cite[p.89, Exercise 11]{CLO:IVAA4}).
  Since then all
  elements of $I$ have zero remainder in every standard expression
  by $G$,
  we obtain $f \not\in I$, and therefore, $f$ does not vanish on
  all points in $S_1 \times \cdots \times S_n$.
\end{proof}

\section{A Nullstellensatz for structured grids using conditions on the monomials
         with multiplicity} \label{sec:cnv-mult}
In this section, we will prove Theorem~\ref{thm:cnv-mult} and
Corollary~\ref{cor:nica-mult}.
Let $\ob{K}$ be a field, and let $t \in \N$. We define
$\I_t (X)$ as the set of all $f \in \kxn$
that have a $t$-fold zero at each $\vb{a} \in X$. Hence
$\I_0 (X) = \kxn$ since every place is a $0$-fold zero of every polynomial,
and $\I_1 (X) = \I (X)$. 
For $X$ being a grid, \cite{BS:PCN} provides a basis
of $\I_t (X)$. For arbitrary finite $X$,
generators of $\I_t(X)$ can be determined from generators
of $\I(X)$ using some arguments on
ideals in commutative rings that we collect in the next
two lemmata. For an ideal $I$ of $\kxn$ and $t \in \N$,
$I^t$ denotes the $t$-th power of the ideal $I$, which is
defined to be the ideal generated by all products
$i_1 \cdots i_t$ with (not necessarily distinct) $i_1, \ldots, i_t \in I$. The product
of two ideals $I,J$ is the ideal generated
by $\{ij \mid i \in I, j \in J\}$ and denoted by $IJ$.
\begin{lem} \label{lem:mt}
  Let $R$ be a commutative ring with unit, let $s, t \in \N$, and
  let $M_1, \ldots, M_s$ be distinct maximal ideals of $R$.
  Then $(\bigcap_{i \in \ul{s}} M_i)^t = \bigcap_{i \in \ul{s}} M_i^t$.
\end{lem}
\begin{proof}
  We proceed by induction on $s$. For $s=1$, the statement is obvious.
  For the induction step, let $s \ge 2$, and let
  $J := \bigcap_{i = 2}^s M_i$.
  First, we observe that for any
  collection $I_1, \ldots, I_k$ of ideals with
  $I_i \not\subseteq M_1$ for all $i \in \ul{k}$,
  we have
  \begin{equation} \label{eq:pM1}
    I_1 I_2 \cdots I_k \not\subseteq M_1.
  \end{equation}  
  In order to show~\eqref{eq:pM1}, we observe that
  for $i \in \ul{s}$ with $i \ge 2$, the assumption $I_i \not\subseteq M_1$
  implies that there
  is $a_i \in I_i \ohne M_1$. The ideal $M_1$ is maximal, and therefore prime,
  and thus $\prod_{i = 1}^k a_i \not\in M_1$ and $\prod_{i=1}^k a_i \in I_1 \cdots I_k$. 
  This proves~\eqref{eq:pM1}.   
  Setting $k := s-1$ and $I_j := M_{j+1}$, we obtain
  $M_2\cdots M_s \not\subseteq M_1$, and since $M_2 \cdots M_s \subseteq J$ also
  \begin{equation} \label{eq:JM1}
  J \not\subseteq M_1.
  \end{equation}
  Next, we show that for all ideals $I$ of $R$ with $I \not\subseteq M_1$
  and for all $r \in \N$, we have
  \begin{equation} \label{eq:IM1}
     I \cap M_1^r =  I M_1^r.
  \end{equation}
  The $\supseteq$-inclusion is obvious, so we only prove
  $\subseteq$.
  Let $g \in I \cap M_1^r$. Since $I \not\subseteq M_1$, we have
  $I + M_1 = R$ and thus there are $a \in I$ and  $b \in M_1$ such that
  $1 = a + b$. Then
  $g = (a + b)^r g = b^r g + \sum_{i=1}^{r} {r \choose i} a^i b^{r-i} g$.
  Then $b^r g \in M_1^r I = I M_1^r$, and 
  for $i \ge 1$, we have
  $a^i \in I$ and $g \in M_1^r$, and therefore
  $a^i b^{r-i} g \in IM_1^r$. Thus 
  $g \in IM_1^r$,  which establishes~\eqref{eq:IM1}.
  Now $\bigcap_{i \in \ul{s}} M_i^t =
  M_1^t \, \cap \, \bigcap_{i = 2}^s M_i^t$. 
  By the induction hypothesis, the last expression is equal to
  $M_1^t \cap (\bigcap_{i = 2}^s M_i)^t = M_1^t \cap J^t$.
  By~\eqref{eq:JM1} and~\eqref{eq:pM1}, we have $J^t \not\subseteq M_1$,
  and thus by~\eqref{eq:IM1}, we have
  $M_1^t \cap J^t = M_1^t J^t = (M_1 J)^t$.
  By \eqref{eq:JM1} and~\eqref{eq:IM1}, $(M_1 J)^t =
  (M_1 \cap J)^t =
  (\bigcap_{i=1}^s M_i)^t$. 
\end{proof}
\begin{lem} \label{lem:im}
  Let $X$ be a finite subset of $\kn$, and let $t \in \N$. Then
  $\I_t (X) = \I(X)^t$.
\end{lem}
\begin{proof}
  For $\vb{a} \in X$, let $M_{\vb{a}} := \I (\{ \vb{a} \})$.
    Then $\I(X) = \bigcap_{\vb{a} \in X} M_{\vb{a}}$.
    We observe that by definition, $\vb{a}$ is a $t$-fold zero of $f$
    if and only if $f (\vb{x} + \vb{a})$ lies in the ideal
    $\XI^t$. Applying the isomorphism $\sigma : \kxn \to \kxn$,
    with $\sigma (x_i) := x_i - a_i$, we obtain that
    $f (\vb{x} + \vb{a}) \in \XI^t$ if and only if
    $f (\vb{x}) \in \langle x_1 - a_1, \ldots, x_1 - a_n \rangle^t$,
    which is equivalent to $f(\vb{x}) \in M_{\vb{a}}^t$.
    Hence 
    $\I_t (X) = \bigcap_{\vb{a} \in X} M_{\vb{a}}^t$.
    Now Lemma~\ref{lem:mt} yields the required equality
    $\I(X)^t = \I_t (X)$.
\end{proof}
As a consequence, we obtain a set of generators of
$\I_t (S)$ for a grid $S$. For a finite subset
$G$ of $\kxn$, we define $G^t := \{ g_1 \cdots g_t \mid
g_1, \ldots, g_t \in G \}$. The ideal of $\kxn$ that is
generated by $G$ is denoted by $\submod{G}$.
\begin{lem}[{\cite[Theorem~3.1]{BS:PCN}}] \label{lem:gtit}
  Let $S = \bigtimes_{i=1}^n S_i$ be a grid over $\kk$,
  for each $i \in \ul{n}$, let $g_i := \prod_{a \in S_i} (x_i - a)$,
  and let $G := \{g_1, \ldots, g_n\}$
  Then $G^t$ generates the ideal $\I_t (S)$.
\end{lem}
\begin{proof}
  By \cite[Theorem~1.1]{Al:CN}, $G$ generates
  the ideal $\I(S)$. Therefore, $G^t$ generates the ideal
  $\I(S)^t$, which is equal to $\I_t (S)$ by Lemma~\ref{lem:im}.
\end{proof}
Lemma~\ref{lem:gtit} yields that every $f \in \I_t (S)$ can be written
as $f = \sum_{i=1}^k h_i g'_i$ with $k \in \N_0$,
$h_1, \ldots, h_k \in \kxn$ and $g'_1, \ldots, g'_k \in G^t$.
As an additional piece of information, 
\cite[Theorem~3.1]{BS:PCN} ensures that we can pick these summands in a way
that $\deg (h_k g'_k) \le \deg (f)$. We give an alternative
argument for
these degree bounds by showing that $G^t$ is a Gr\"obner basis.
For this purpose, we extend Buchberger's First Criterion
\cite[Lemma~5.66]{BW:GB}.
\begin{thm} \label{thm:gt}
  Let $s,t \in \N$, and let $g_1, \ldots, g_s \in \kxn \setminus \{0\}$ be such that
  for $i,j \in \ul{s}$ with $i \neq j$, $\LM(g_i)$ and $\LM(g_j)$
  do not have any variable in common, i.e.,
  $\gcd (\LM(g_i), \LM(g_j)) = 1$, and let $\lea$ be an admissible
  ordering of monomials.
  Then \(G^t := \{ g_1^{\alpha_1} \cdots g_s^{\alpha_s} \mid
                  \alpha_1, \ldots, \alpha_s \in \N_0, \sum_{i=1}^s
                  \alpha_i = t \}\)
  is a Gr\"obner basis of the ideal $\submod{G}^t$ with respect
  to $\lea$.
\end{thm}
\begin{proof}
  Clearly, the set $G^t$ generates the ideal $\submod{G}^t$.
  We now show that $G^t$ is a Gr\"obner basis. To this end, we
  use Buchberger's Characterization Theorem for Gr\"obner bases
  \cite[Theorem~5.64]{BW:GB}
  (cf. \cite[Theorem~6.2]{Bu:GBAA}, \cite[Theorem~15.8]{Ei:CA}) 
  which states that
  $G^t$ is a Gr\"obner basis if for all $f,h \in G^t$,
  the $S$-polynomial $S(f,h)$ has a standard
   expression by $G^t$ with remainder~$0$. 
   Without loss of generality, we assume that $g_1, \ldots, g_s$ are
   normed, i.e., $\LC (g_i) = 1$ for all $i \in \ul{s}$, and 
   we show that every $S$-polynomial of $G^t$ has a standard
   expression by $G^t$
   with remainder~$0$.
   To this end, let
   $f = g_1^{\alpha_1} \cdots g_s^{\alpha_s}$ and $h = g_1^{\beta_1} \cdots
        g_s^{\beta_s}$ be elements of  $G^t$.
   The $S$-polynomial $S(f,h)$ can be computed as
   \[
       S(f,h) = \frac{\lcm (\LM (f), \LM(h))}{\LM(f)} \, f
   -
    \frac{\lcm (\LM (f), \LM(h))}{\LM(h)}\, h.
   \]
   We have
   \[
      \begin{split}
   \lcm (\LM(f), \LM(h)) &=
   \lcm \big(\LM (\prod_{i=1}^s g_i^{\alpha_i}),
         \LM (\prod_{i=1}^s g_i^{\beta_i})\big) \\ &=
    \lcm \big(\prod_{i=1}^s \LM (g_i)^{\alpha_i},
          \prod_{i=1}^s \LM (g_i)^{\beta_i}\big) \\  &=
          \prod_{i=1}^s \LM(g_i)^{\max (\alpha_i, \beta_i)},
      \end{split}
   \]   
          where
          the last equality holds because the $\LM(g_i)$
    are coprime.
    Now let $I = \{i \in \ul{s} \mid \alpha_i > \beta_i\}$
    and $J = \{j \in \ul{s} \mid \alpha_j < \beta_j\}$.
    Then
    \[
    S(f,h) = \big(\prod_{j \in J} \LM(g_j)^{\beta_j - \alpha_j}\big) \,
    g_1^{\alpha_1} \cdots g_s^{\alpha_s} -
     \big(\prod_{i \in I} \LM(g_i)^{\alpha_i - \beta_i}\big) \,
     g_1^{\beta_1} \cdots g_s^{\beta_s}.
   \]
   Since $(\prod_{j \in J} g_j^{\beta_j - \alpha_j}) \, g_1^{\alpha_1} \cdots g_s^{\alpha_s}
   = (\prod_{i \in I} g_i^{\alpha_i - \beta_i}) \,  g_1^{\beta_1} \cdots g_s^{\beta_s}$,
   we have
   \begin{multline} \label{eq:candrep}
   S(f,h) =
   \bigl(
   \prod_{i \in I} g_i^{\alpha_i - \beta_i}
     -
   \prod_{i \in I} \LM(g_i)^{\alpha_i - \beta_i}
     \bigr)
    g_1^{\beta_1} \cdots g_s^{\beta_s}   
    \\ -
   \bigl(
   \prod_{j \in J} g_j^{\beta_j - \alpha_j}
     -
   \prod_{j \in J} \LM(g_j)^{\beta_j - \alpha_j}
   \bigr)
   g_1^{\alpha_1} \cdots g_s^{\alpha_s}.
   \end{multline}
   We show that this is a standard expression of $S(f,h)$ by $G^t$.
   If at least one of the two summands in the right hand side of~\eqref{eq:candrep} is $0$, then we are done.
   Otherwise we establish
   that the two summands have different multidegree.
   Seeking a contradiction, we suppose that they have the same multidegree.
   Then since $\prod_{i \in I} \LM(g_i)^{\alpha_i}$ divides the leading monomial
   of the second summand in the right hand side of~\eqref{eq:candrep},
   it also divides the leading monomial of the first summand.
   Using that for $i \in I$ and $j \not\in I$,
   $\LM (g_i)^{\alpha_i}$ is coprime
   to $\LM(g_j)^{\beta_j}$, we obtain that
   \begin{equation} \label{eq:div}
      \prod_{i \in I} \LM (g_i)^{\alpha_i} \text{ divides }
     \LM  \Big(\bigl(
   \prod_{i \in I} g_i^{\alpha_i - \beta_i}
     -
   \prod_{i \in I} \LM(g_i)^{\alpha_i - \beta_i}
     \bigr)
     \prod_{i \in I} \LM (g_i)^{\beta_i} \Big).
   \end{equation}  
   We have 
     $\LM\big(  \prod_{i \in I} g_i^{\alpha_i - \beta_i}
              -
              \prod_{i \in I} \LM(g_i)^{\alpha_i - \beta_i} \big)
              <_a \prod_{i \in I} \LM(g_i)^{\alpha_i - \beta_i}$,
   and therefore the degree of the right hand side
   is less (w.r.t $<_a$) than
   the degree of the left hand side of~\eqref{eq:div}.
   Since the right hand side is not $0$, this is a contradiction.
\end{proof}
Lemma~\ref{lem:gtit} and Theorem~\ref{thm:gt} allow us to determine a Gr\"obner basis
for the ideal $\I_t (S)$ associated with a grid $S$.
After these preparations, we can now prove the main results
established in this section.
    \begin{proof}[Proof of Theorem~\ref{thm:cnv-mult}]
  For each $i \in \ul{n}$, 
  let $g_i := \prod_{a \in S_i} (x_i - a)$. Let
  $G = \{g_1, \ldots, g_n \}$, and 
  $I$ be the ideal of $\kxn$ given
  by $I := \submod{G}^t$. 
    By Lemma~\ref{lem:gtit}
  all elements of $S$ are $t$-fold zeros $f$
  if and only if $f \in I$. Now we seek to apply
  Corollary~\ref{cor:rem}.
  For each $i \in \ul{n}$,  $S_i$ is $\lambda_i$-null
  and thus the polynomial $g_i$ 
  is $\lambda$-lacunary. Therefore, by Lemma~\ref{lem:proquolac}
  each polynomial in $G^t$ is $\lambda$-lacunary.
  We will now show that $\mono{\alpha}$ is a $(G^t, \lambda)$-stable monomial
  in $f$.
  First, suppose that there is $h \in G^t$ such that
  $\LM (h)$ divides $\mono{\alpha}$. Let $\beta_1, \ldots, \beta_n \in \N_0$
  such that $\sum_{i=1}^n \beta_i = t$ and $h = g_1^{\beta_1} \cdots
  g_n^{\beta_n}$. Then $\LM (h) = \prod_{i=1}^n \LM(g_i)^{\beta_i} =
  \prod_{i=1}^n x_i^{\beta_i |S_i|}$ and thus, we have
  $\beta_i |S_i| \le \alpha_i$ for all $i \in \N$, contradicting
  the assumption stated in~\eqref{eq:betat}.  
   Therefore the monomial
  $\LM(h)$ does not divide $\mono{\alpha}$, which establishes
  Condition~\eqref{it:st22} of Definition~\ref{de:stable2}.
  Next, we show that $f$ contains no $(G^t, \lambda, \alpha)$-shading
  monomial.
  Let $\mono{\gamma} \in \Mon (f)$. If $\alpha = \gamma$, then
  $\mono{\gamma}$ violates Condition~\eqref{it:sha1} of
  Definition~\ref{de:shading2} and is therefore not
  $(G^t,\lambda,\alpha)$-shading.
  If $\gamma \neq \alpha$, the assumption yields
  an $i \in \ul{n}$ such that $\gamma_i \neq \alpha_i$ and~\eqref{eq:aich3mult} holds.
  If $\gamma_i \in [0, \alpha_i - 1]$, then Condition~\eqref{it:sha1}
  of Definition~\ref{de:shading2} is violated, and so $\mono{\gamma}$
  is not $(G^t,\lambda,\alpha)$-shading.
  If $\gamma_i \in [\alpha_i + 1, \alpha_i + \lambda_i]$, then
  Condition~\eqref{it:sha3} of 
  of Definition~\ref{de:shading2} is violated and
  thus
  case $\mono{\gamma}$ is not $(G^t,\lambda,\alpha)$-shading.
  Now we turn to the case that the last alternative in~\eqref{eq:aich3mult}
  holds and that $\gamma_i > \alpha_i + \lambda_i$.
  Seeking a contradiction, we suppose that
  Condition~\eqref{it:sha2} of Definition~\ref{de:shading2} is satisfied.
  This condition tells that there is $\beta \in \NON$ such that
  $\sum_{j=1}^n \beta_j = t$ and 
  $\LM (g_1^{\beta_1} \ldots g_n^{\beta_n}) \mid \mono{\gamma}$ and
  $\beta_i |S_i| > 0$.
   Since $\LM (g_1^{\beta_1} \ldots g_n^{\beta_n}) \mid \mono{\gamma}$,
   we have $\beta_j |S_j| \le \gamma_j$ for all $j \in \ul{n}$. This contradicts  the case assumption
  that the last alternative in~\eqref{eq:aich3mult} holds.
  Hence Condition~\eqref{it:sha2} of Definition~\ref{de:shading2} fails,
  and thus 
  $\mono{\gamma}$  is not $(G^t,\lambda,\alpha)$-shading.

  Therefore, $f$ contains no $(G^t, \lambda, \alpha)$-shading
  monomial and thus, $\mono{\alpha}$ is $(G^t, \lambda)$-stable.
    Let
   \(
      f = \sum_{j=1}^t c_j \, \mono{\delta_j} g_{i_j} + r
  \)
   be a natural standard expression of $f$ by $G$.
   Since $\mono{\alpha}$ is a $(G^t, M)$-stable monomial in $f$,
   Corollary~\ref{cor:rem} yields that 
    $\mono{\alpha} \in \Mon(r)$. 
   Since the leading monomials of the polynomials in $G$ are coprime,
   Theorem~\ref{thm:gt} yields that 
   the set $G^t$ is a Gr\"obner basis of $I$.
     Since then all
  elements of $I$ have zero remainder in every standard expression
  by $G^t$,
  we obtain $f \not\in I$, and therefore, not
  all points in $S_1 \times \cdots \times S_n$ are $t$-fold
  zeros of $f$.
\end{proof}

    \begin{proof}[Proof of Corollary~\ref{cor:nica-mult}]
      We assume that $\sum_{i=1}^n \alpha_i \ge \deg (f) - \lambda$,
      and 
      we show that for $\lambda_1 = \cdots = \lambda_n = \lambda$,
      the assumptions of
  Theorem~\ref{thm:cnv-mult} are satisfied.
  Suppose that the assumption on the monomials
  $\mono{\gamma} \in \Mon (f) \setminus \{\mono{\alpha}\}$ fails.
  Then there is a monomial $\mono{\gamma} \in \Mon (f)$ with
  $\gamma \neq \alpha$ such that for all $i \in \ul{n}$, we have
  $\gamma_i = \alpha_i$ or $\gamma_i > \alpha_i + \lambda$.
  Then  $\deg (f) \ge \deg (\mono{\gamma}) > (\sum_{i=1}^n \alpha_i) + \lambda$, 
  contradicting the assumption~\eqref{eq:nica2} of
  Corollary~\ref{cor:nica-mult}.
  Now Theorem~\ref{thm:cnv-mult} yields the result.
 \end{proof}

    \section{A Nullstellensatz for grids of multisets} \label{sec:multisets}
    In this section, we add two improvements to K\'os and
    R\'onyai's Nullstellensatz for multisets given in
    \cite[Theorem~6]{KR:ANFM}: we consider \emph{structured} grids,
    and we give a generalization in terms of \emph{conditions on the monomials}
    in the spirit of \cite{Sc:ASPD}. 
    Let $f \in \kxn$, and let $(m_1, \ldots, m_n) \in \NON$.
    We say that $\vb{c}$ is a zero with multiplicity vector
    $(m_1, \ldots, m_n)$
    if
    $f(\vb{x} + \vb{c}) \in \langle x_1^{m_1}, \ldots, x_n^{m_n} \rangle$.
    Let $T$ be a mapping from the finite set $U$ to $\N_0$. Then we also call
    $T$ a \emph{multiset} and $U$ the \emph{domain} of $T$.
    Let $S_1, \ldots, S_n$ be multisets with domains
    $U_1, \ldots, U_n$.
    For $n \in \N$, the \emph{multigrid} $S$ denoted by
    $\bigtimes_{i=1} S_i$ is a mapping with domain
    $U_1 \times \cdots \times U_n$, codomain $\NON$
    and $S (u_1, \ldots, u_n) := (S_1 (u_1), \ldots, S_n (u_n))$.
    Suppose that $U_1, \ldots, U_n$ are subsets of a field
    $\ob{K}$.
    Then we say that $f \in \kxn$ \emph{vanishes on the multigrid $S$} if
    for all $\vb{u} = (u_1, \ldots, u_n) \in \bigtimes_{i=1}^n U_i$,
    $\vb{u}$ is a zero with multiplicity vector
    $S (\vb{u}) = (S_1 (u_1), \ldots, S_n (u_n))$.
    The following result is a consequence of \cite[Theorem~1]{KR:ANFM}.
    \begin{thm} \label{thm:multiset}
      Let $S$ be a multigrid with domain $\bigtimes_{i=1}^n U_i$,
      and let $f \in \kxn$. Then
      $f$ vanishes on $S$ if and only if
      $f$ lies in the ideal $\submod{g_1, \ldots, g_n}$,
      where $g_i := \prod_{u \in U_i} (x_i - u)^{S_i (u)}$.
    \end{thm}
    We say that a multiset $S_i : U_i \to \N_0$ is
    $\lambda_i$-null if $\prod_{u \in U_i} (x_i - u)^{S_i (u)}$ is
    $\lambda_i$-lacunary.
    Now our generalization of \cite[Theorem~6]{KR:ANFM} is:
    \begin{thm}[Structured Nullstellensatz for multisets using conditions on the monomials] \label{thm:multisetscnii}
      Let $S = \bigtimes_{i=1}^n S_i$ be a multigrid with
      domain $\bigtimes_{i=1}^n U_i$,
      let $f \in \kxn$, and let $\alpha = (\alpha_1, \ldots, \alpha_n)$ and
      $\lambda = (\lambda_1, \ldots, \lambda_n)$ be elements of $\NON$.
      We assume that each $S_i$ is $\lambda_i$-null.
       Let $||S_i|| := \sum_{u \in U_i} S_i (u)$.
       We assume that $\alpha_i < ||S_i||$ for all $i \in \ul{n}$ and
       that $f$ contains the monomial 
       $x_1^{\alpha_1} \cdots x_n^{\alpha_n}$.
   Furthermore,
  we assume that for every
  monomial $\mono{\gamma}$ in $\Mon (f) \setminus \{\mono{\alpha}\}$,
  there is $i \in \ul{n}$ such that
    \begin{equation} \label{eq:aich3multiset}
      \gamma_i \in [0, \alpha_i - 1] \union [\alpha_i + 1, ||S_i|| - 1]
               \union [||S_i||, \alpha_i + \lambda_i].
     \end{equation}
  Then $f$ does not vanish on the multigrid $S$.
    \end{thm}
    \begin{proof}
   For each $i \in \ul{n}$, 
  let $g_i := \prod_{u \in U_i} (x_i - u)^{S_i (u)}$. Then
  the leading monomial of $g_i$ is $x_i^{||S_i||}$. Let
  $I$ be the ideal of $\kxn$ generated by
  $G = \{g_1, \ldots, g_n \}$. By Theorem~\ref{thm:multiset},
  a polynomial $f$ vanishes on the multigrid $S_1 \times \cdots \times S_n$
  if and only if it lies in $I$.
  The remainder of the proof is an almost verbatim copy of
  the proof
  of Theorem~\ref{thm:cnv}; again, we seek to apply
  Corollary~\ref{cor:rem}.
  For each $i \in \ul{n}$, $S_i$ is $\lambda_i$-null
  and therefore the polynomial $g_i$
  is $(\lambda_1, \ldots, \lambda_n)$-lacunary.
  We will now show that $\mono{\alpha}$ is a $(G,\lambda)$-stable monomial
  in $f$ with respect to $\lea$.
  First, we observe that for each $i \in \ul{n}$ we have
  $\alpha_i < ||S_i||$ and therefore the monomial
  $\LM(g_i)$ does not divide $\mono{\alpha}$, which establishes
  Condition~\eqref{it:st22} of Definition~\ref{de:stable2}.
  Next, we show that $f$ contains no $(G, \lambda, \alpha)$-shading
  monomial.
  Let $\mono{\gamma} \in \Mon (f)$. If $\alpha = \gamma$, then
  $\mono{\gamma}$ violates Condition~\eqref{it:sha1} of
  Definition~\ref{de:shading2} and is therefore not
  $(G,\lambda,\alpha)$-shading.
  If $\gamma \neq \alpha$, the assumption yields
  an $i \in \ul{n}$ such that
  $$\gamma_i \in [0, \alpha_i - 1] \union [\alpha_i + 1, ||S_i|| - 1] \union
           [||S_i||, \alpha_i + \lambda_i].$$
  If $\gamma_i \in [0, \alpha_i - 1]$, then Condition~\eqref{it:sha1}
  of Definition~\ref{de:shading2} is violated, and so $\mono{\gamma}$
  is not $(G,\lambda,\alpha)$-shading.
  If $\gamma_i \in [\alpha_i + 1, ||S_i|| - 1]$ and Condition~\eqref{it:sha2}
  of Definition~\ref{de:shading2} is satisfied, then
  $\LM(g_i) \mid \mono{\gamma}$, and therefore
  $|S_i| \le \gamma_i$, contradicting $\gamma_i \le ||S_i|| - 1$. We
  conclude that also in the case $\gamma_i \in [\alpha_i + 1, ||S_i|| - 1]$,
  $\mono{\gamma}$  is not $(G,\lambda,\alpha)$-shading.
  If $\gamma_i \in [||S_i||, \alpha_i + \lambda_i]$, then
  we have $\gamma_i \ge ||S_i|| > \alpha_i$. If Condition~\eqref{it:sha3}
  of Definition~\ref{de:shading2} is satisfied, we have
  $\gamma_i > \alpha_i + \lambda_i$, contradicting
  $\gamma_i \le \alpha_i + \lambda_i$. Hence also in this
  case $\mono{\gamma}$ is not $(G,\lambda,\alpha)$-shading.
  Therefore, $f$ contains no $(G, \lambda, \alpha)$-shading
  monomial. Therefore, $\mono{\alpha}$ is $(G, \lambda)$-stable.
  
     Let
   \(
      f = \sum_{j=1}^t c_j \, \mono{\delta_j} g_{i_j} + r
  \)
   be a natural standard expression of $f$ by $G$.
   Since $\mono{\alpha}$ is a $(G, \lambda)$-stable monomial in $f$,
   Corollary~\ref{cor:rem} yields that 
    $\mono{\alpha} \in \Mon(r)$. 
   Since the leading monomials of the polynomials in $G$ are coprime,
   \cite[Lemma~5.66]{BW:GB} yields that
   the set $G$ is a Gr\"obner basis of $I$
   (cf. \cite[p.89, Exercise 11]{CLO:IVAA4}).
  Since then all
  elements of $I$ have zero remainder in every standard expression
  by $G$,
  we obtain $f \not\in I$, and therefore, $f$ does not vanish on
  the multigrid $S$.
\end{proof}

    \section{Nullstellens\"atze for punctured and structured grids}
    \label{sec:punct}
    In this Section, we prove Theorem~\ref{thm:ballcn2} and
    Corollary~\ref{cor:ballcn2cor}. We start with the description
    of the vanishing ideal of a punctured grid. We note that
    a similar result that includes multiplicities has been
    given in~\cite[Theorem~4.1]{BS:PCN}, but our proof is
    different. For a subset $J$ of $\kxn$ and $f \in \kxn$, we write
    $\V(J)$ for the set $\{\vb{a} \in \kn \mid f (\vb{a}) = 0
    \text{ for all } f \in J \}$ of common zeros of $J$.
\begin{thm}[cf. {\cite[Theorem~4.1]{BS:PCN}}]  \label{thm:ballcn1}
  Let $X \setminus Y = (\bigtimes_{i=1}^n X_i) \setminus
  (\bigtimes_{i=1}^n Y_i)$ be a punctured grid with
  $Y_i \subseteq X_i$ for all $i \in \N$,
  let $g_i := \prod_{a \in X_i} (x_i - a)$, let
  $l_i := \prod_{a \in Y_i} (x_i - a)$, and let
  $f \in \kxn$. Then $f$ vanishes on $X \setminus Y$
  if and only if
  $f \in \submod{g_1, \ldots, g_n, \prod_{i=1}^n \frac{g_i}{l_i}}$.
  Furthermore, for every admissible monomial ordering $\lea$,
  $\{g_1, \ldots, g_n, \prod_{i=1}^n \frac{g_i}{l_i} \}$
  is a Gr\"obner basis  with respect to $\lea$.
\end{thm}    
\begin{proof}
  We first compute generators of the vanishing ideal
  $\I (X \ohne Y)$.
  We have $X \setminus Y
  =
  X \setminus \bigcap_{i \in \ul{n}} X_1 \times \cdots \times X_{i-1}
  \times Y_i \times X_{i+1}  \times  \cdots \times X_n
  =
  X \cap \, \bigcup_{i \in \ul{n}} X_1 \times \cdots \times X_{i-1}
  \times (X_i \setminus Y_i) \times X_{i+1}  \times  \cdots \times X_n
  =
  X \cap \, \bigcup_{i \in \ul{n}}  \ob{K} \times \cdots \times \ob{K}
  \times (X_i \setminus Y_i) \times \ob{K} \times  \cdots \times \ob{K}
  =
  X \cap \, \bigcup_{i \in \ul{n}} \V (\frac{g_i}{l_i})
  =
  X \cap \, \V (\prod_{i \in \ul{n}} \frac{g_i}{l_i})$.
  Now for finite $X$, Clark's Finitesatz \cite[Theorem~7]{Cl:TCNR}
  tells that for any ideal
  $J$ of $\kxn$, we have $\I (X \cap \V(J)) = \I(X) + J$, and therefore
  \[
  \I (X \setminus Y) 
  = \I (X \cap \, \V (\prod_{i \in \ul{n}} \frac{g_i}{l_i}))
  =
  \I(X) + \langle
                          \prod_{i \in \ul{n}} \frac{g_i}{l_i}
                          \rangle =
                  \langle g_1, \ldots, g_n, \prod_{i \in \ul{n}} \frac{g_i}{l_i}
                  \rangle.
  \]
  In order to show that $G := \{ g_1, \ldots, g_n, \prod_{i \in \ul{n}} \frac{g_i}{l_i} \}$ is a Gr\"obner basis, we again use
  Buchberger's Characterization Theorem for Gr\"obner bases \cite[Theorem~5.64]{BW:GB},
  and we therefore look at the $S$-polynomials of $G$.
  First, we note that for $i, j \in \ul{n}$ with
  $i \neq j$, the leading monomials of $g_i$ and $g_j$ are coprime,
  and therefore the $S$-polynomial of $g_i$ and $g_j$ has a standard
  expression by $\{g_i, g_j\}$  with remainder $0$ (\cite[Lemma~5.66]{BW:GB});
  such an expression can
  also be obtained by setting $t := 1$ in the proof of Theorem~\ref{thm:gt}.
  For checking the other $S$-polynomials,
  let $j \in \ul{n}$, and let $p := \prod_{i \in \ul{n}} \frac{g_i}{l_i}$.
  We compute $S (g_j, p)$. Let $a_i := |X_i|$ and $b_i := |Y_i|$.
  Then
  \[
       \LM (g_j) = x_j^{a_j} \text{ and }
       \LM (p) = \prod_{i \in \ul{n}} x_i^{a_i - b_i}.
  \]
  Now let $q := \prod_{i \in \ul{n} \setminus \{j\}} \frac{g_i}{l_i}$; then
  $q \frac{g_j}{l_j} = p$. Then we have
   \begin{equation*}
   S (g_j, p) = (\prod_{i \in \ul{n} \setminus \{j\}} x_i^{a_i - b_i}) g_j
                -
                x_j^{b_j} p 
                =
                \LM (q) \, g_j
                - \LM (l_j) \, p.
   \end{equation*}                                
   Since $l_j p = q  g_j$,
   we have 
   \[
           \LM (q) \, g_j
           - \LM (l_j) \, p =
           (l_j - \LM (l_j)) \, p -
           (q - \LM(q)) \, g_j.           
    \]
    We show that $S(g_j , p) = (l_j - \LM (l_j)) \, p - (q - \LM(q)) \, g_j$
    is a standard expression of $S(g_j, p)$ by $\{g_j, p\}$.
  If at least one of the summands is $0$, we are done.
  Otherwise, we show 
    that the two summands have different multidegree.
     Suppose that both summands are nonzero and that they have the same multidegree. Then
    $\LM (g_j)$ divides $\LM ((l_j - \LM (l_j)) \, p)$.
    We have $\deg_{x_j} (\LM(g_j)) = a_j$ and
    $\deg_{x_j} (\LM ( (l_j - \LM (l_j))\, p))
    = \deg_{x_j} (\LM (l_j - \LM (l_j)) \, \LM (p))
    = \deg_{x_j} (\LM (l_j - \LM (l_j))) + \deg_{x_j} (\LM (p))
    \le \deg_{x_j} (\LM (l_j - \LM (l_j))) + \deg_{x_j} (p)
    \le (b_j - 1) + (a_j - b_j) =
    a _j - 1$.
    But then $\LM (g_j) \nmid \LM ((l_j - \LM (l_j))\, p)$.

    Since all $S$-polynomials have a standard expression by $G$ with
    remainder $0$, Buchberger's Characterization Theorem yields
    that $G$ is a Gr\"obner basis.
\end{proof}

\begin{proof}[Proof of Theorem~\ref{thm:ballcn2}]
  Let $p := \prod_{i=1}^n \frac{g_i}{l_i}$, and  let
  $G' := \{g_1, \ldots, g_n, p \}$.
  We first show that $\mono{\alpha}$ is a $(G',\lambda)$-stable monomial
  in $f$.
  Let us first show that there is no $g \in G'$ with
  $\LM(g) \mid \mono{\alpha}$. If there is $i \in \ul{n}$ with
  $\LM (g_i) \mid \mono{\alpha}$, then $|X_i| \le \alpha_i$,
  contradicting Assumption~\eqref{it:al1}.
  If $\LM(p) \mid \mono{\alpha}$, then
  we have $|X_i| - |Y_i| \le \alpha_i$  for all $i \in \ul{n}$,
  contradicting Assumption~\eqref{it:al2}.
  It remains to show that $f$ contains no $(G, \lambda, \alpha)$-shading
  monomial. Let $\mono{\gamma} \in \Mon (f)$. By the assumptions,
  there is $i \in \ul{n}$ such that one of the Conditions~\eqref{it:al3a},
  \eqref{it:al3b}, \eqref{it:al3c} holds.
  We fist consider the case that $\gamma_i \in [0, \alpha_i - 1]$.
  Then Condition~\eqref{it:sha1} of Definition~\ref{de:shading2} fails,
  and thus $\mono{\gamma}$ is not $(G', \lambda, \alpha)$-shading.
  In the case $\gamma_i \in [\alpha_i + 1, \alpha_i + \lambda_i]$,
  Condition~\eqref{it:sha3} of Definition~\ref{de:shading2} fails,
  and thus $\mono{\gamma}$ is not $(G', \lambda, \alpha)$-shading.
  In the case that $\gamma_i \in [\alpha_i + 1, |X_i| - 1]$ and $|X_i| = |Y_i|$,
  we show that Condition~\eqref{it:sha2} of Definition~\ref{de:shading2} fails.
  Seeking a contradiction, we suppose that Condition~\eqref{it:sha2}
  holds. Then there is $g \in G'$ with $\deg_{x_i} (g) > 0$ and
  $\LM(g) \mid \mono{\gamma}$. Since $|X_i| = |Y_i|$, $\deg_{x_i} (p) = 0$
  and thus $g = g_i$. Hence $|X_i| \le \gamma_i$, contradicting the assumptions.
  Thus Condition~\eqref{it:sha2} of Definition~\ref{de:shading2} fails,
  and $\gamma$ is not $(G', \lambda, \alpha)$-shading.
  Now we consider the case that $\gamma_i \in [\alpha_i + 1, |X_i| - 1]$ and that there is $j \in \ul{n}$ with $\gamma_j < |X_j| - |Y_j|$.
  Supposing again that $\gamma$ is $(G',\lambda, \alpha)$-shading,
  we obtain that $\LM(g_i) \mid \mono{\gamma}$ or $\LM(p) \mid \mono{\gamma}$.
  If $\LM (g_i) \mid \mono{\gamma}$, then $|X_i| \le \gamma_i$, contradicting
  the assumptions. If $\LM (p) \mid \mono{\gamma}$, then
  for all $j \in \ul{n}$, we have $|X_j| - |Y_j| \le \gamma_j$. This is
  also excluded by the case assumption.
  
  Therefore $\mono{\alpha}$ is a $(G',\lambda)$-stable
  monomial in $f$. Let $g_{n+1} := p$ and 
     let
   \(
      f = \sum_{j=1}^t c_j \, \mono{\delta_j} g_{i_j} + r
  \)
  be a natural standard expression of $f$ by $G'$.
  We note that $g_1, \ldots, g_n$ are $\lambda$-lacunary by
  assumption. Furthermore, for each $i \in \ul{n}$,
  $g_i$ and $l_i$ are 
  $\lambda$-lacunary and hence by Lemma~\ref{lem:proquolac},
  $p = \prod_{i=1}^n \frac{g_i}{l_i}$ is $\lambda$-lacunary.
    Since $\mono{\alpha}$ is a $(G', \lambda)$-stable monomial in $f$
  and all $g' \in G'$ are $\lambda$-lacunary,
   Corollary~\ref{cor:rem} yields that 
    $\mono{\alpha} \in \Mon(r)$. 
   By Theorem~\ref{thm:ballcn1}, 
   the set $G'$ is a Gr\"obner basis of $\I (X \setminus Y)$.
     Since then all
  elements of $\I (X \setminus Y)$ have zero remainder in every standard expression
  by $G'$,
  we obtain $f \not\in \I (X \setminus Y)$, and therefore,
  there is $\vb{s} \in X \setminus Y$ with $f (\vb{s}) \neq 0$.
\end{proof}  

\begin{proof}[Proof of Corollary~\ref{cor:ballcn2cor}]
  We show that for $\lambda_1 = \cdots = \lambda_n
  = \lambda$, the assumptions of
  Theorem~\ref{thm:ballcn2} are satisfied.
  Assume that Assumption~\eqref{it:al3} of Theorem~\ref{thm:ballcn2} fails.
  Then there is a monomial $\mono{\gamma} \in \Mon (f) \ohne \{\alpha\}$
  such that for all $i \in \ul{n}$,
  $\gamma_i \not\in [0, \alpha_i - 1] \union [\alpha_i + 1, \alpha_i + \lambda_i]$, which means that for all $i \in \ul{n}$, $\gamma_i = \alpha_i$ or
  $\gamma_i > \alpha_i + \lambda_i$.  Then
  $\deg (f) \ge \deg (\mono{\gamma}) > (\sum_{i=1}^n \alpha_i) + \lambda$,
  contradicting Assumption~\eqref{it:al3new} of
  Corollary~\ref{cor:ballcn2cor}.
  Now Theorem~\ref{thm:ballcn2} yields the result.
 \end{proof}

\section{A punctured version of the Alon-F\"uredi Theorem} \label{sec:AF}
In \cite{Cl:ATCW}, Clark gives a proof of the Alon-F\"uredi Theorem for
counting nonzeros of polynomials on grids that is based
on the fact that for a finite set $X \subseteq \kn$,
the vector space dimension of $\kxn/\I(X)$ is equal to
$|X|$, and this dimension
can be determined as the number of monomials $\mono{\alpha}$
that are not leading monomials of any polynomial in $\I(X)$.
We proceed by adapting some methods from~\cite{Cl:ATCW} to punctured grids. 
For a subset $G$ of $\kxn$ and an admissible monomial ordering
$\le_a$, we define two subsets of
$\{\mono{\alpha} \mid \alpha \in \NON \}$.
\[
   \begin{array}{rcl}
    \filter{G} & := & \{ \mono{\alpha} \mid \alpha \in \NON \text{ and } \exists g \in G : \LM (g) \text{ divides }
    \mono{\alpha} \}, \\
    \cofilter{G} & := & \{\mono{\alpha} \mid \alpha \in \NON\} \ohne (\filter{G}) = \\ &  &
    \{ \mono{\alpha} \mid \alpha \in \NON \text{ and }\text{there is no } g \in G \text{
        such that } \LM (g) \text{ divides } \mono{\alpha} \}.
   \end{array}
\]
\begin{thm}[Clark's formula, \cite{Cl:ATCW}] \label{thm:cf}
  Let $\ob{K}$ be a field, and
  let $X$ be a finite subset of $\kn$, and let $f \in \kxn$.
  Then
  \[
     |X \ohne \V (f)| =
     | \cofilter{\I (X)} \, \ohne \, \cofilter{\I(X) + \submod{f}}|.
  \]
\end{thm}
\begin{proof}
  For every finite subset $Y$ of $\kn$, we have
  \begin{equation} \label{eq:sizeY}
    |Y|= |\cofilter{\I (Y)}|.
  \end{equation}
  To prove this, let $\epsi : \kxn \to \kk^Y$, $\epsi (p) := \hat{p}|_Y$,
  where $\hat{p}$ is the function from $\kn$ to $\kk$ induced
  by $p$, and $\hat{p}|_Y$ denotes its restriction to $Y$.
  Then $\ker \epsi = \I (Y)$, and $\epsi$ is surjective because
  every mapping from $Y$ to $\kk$ can be interpolated by a polynomial.
  Since $\epsi$ is a homomorphism of $\kk$-vector spaces, we 
  get $\dim_{\kk} (\kxn / \I(Y)) = \dim_{\kk} (\kk^Y) = |Y|$.
  The vector space dimension of the residue class ring
  $\dim_{\kk} \kxn / \I(Y)$ is equal to $|\cofilter{\I(Y)}|$
  because $\{ \mono{\alpha} + \I(Y) \mid \mono{\alpha} \in \cofilter{\I(Y)} \}$ is
  a basis of this vector space.
  This proves~\eqref{eq:sizeY}.
  Hence
  $|X \cap \V (f)| = |\cofilter{\I (X \cap \V(f))}|$.
  By Clark's Finitesatz \cite[Theorem~7]{Cl:TCNR},
  the last expression is
  equal to $|\cofilter{\I(X) + \submod{f}}|$.
  Thus $|X \ohne \V(f)| =
  |X| - |X \cap \V(f)| =
  |\cofilter{\I(X)}| - |\cofilter{\I(X) + \submod{f}}|$.
  Since $\I(X) \subseteq \I(X) + \submod{f}$, we have
  $\cofilter{\I(X) + \submod{f}} \subseteq \cofilter{\I(X)}$,
  and therefore
  $|\cofilter{\I(X)}| - |\cofilter{\I(X) + \submod{f}}|
   = |\cofilter{\I(X)} \, \ohne \, \cofilter{\I(X) + \submod{f}}|$.
\end{proof}
For $a, b \in \N_0$, we will denote the
interval $\{x \in \N_0 \mid a \le x < b\}$ by $[a, b)$.
  \begin{lem} \label{lem:punctured}
      Let $X = \bigtimes_{i=1}^n X_i, Y = \bigtimes_{i=1}^n Y_i$ be grids over the
      field $\kk$ with $Y_i \subseteq X_i$ for all $i \in \ul{n}$,
      and for each $i \in \ul{n}$,
  let $a_i := |X_i|$ and $b_i := |Y_i|$. Then
  $\cofilter{\I (X \ohne Y)} = \{ \mono{\alpha} \mid
  \alpha \in \bigtimes_{i \in \ul{n}} [0, a_i) \setminus
    \bigtimes_{i \in \ul{n}} [a_i - b_i, a_i) \}$.
\end{lem}
\begin{proof}
  Let $g_i := \prod_{a \in X_i} (x_i - a)$, and 
  $l_i := \prod_{a \in Y_i} (x_i - a)$. Then by
  Theorem~\ref{thm:ballcn1},
  $G' := \{g_1, \ldots, g_n, \prod_{i=1}^n \frac{g_i}{l_i} \}$ is
  a Gr\"obner basis for $\I(X \setminus Y)$.
  Hence $\cofilter { \I(X \setminus Y)} =
         \cofilter {  G' }=$ \\ $
         \cofilter {\{g_1, \ldots, g_n\}} \, \cap \,
         \cofilter {\{\prod_{i=1}^n \frac{g_i}{l_i} \}}
         =
         \cofilter {\{g_1, \ldots, g_n\}} \ohne
         \filter {\{\prod_{i=1}^n \frac{g_i}{l_i} \}}
         = \{ \mono{\alpha} \mid  \alpha \in \bigtimes_{i \in \ul{n}} [0, a_i) \ohne
           \bigtimes_{i \in \ul{n}} [a_i - b_i , a_i) \}$.
\end{proof}  

\begin{thm}[Clark's Monomial Alon-F\"uredi Theorem \cite{Cl:ATCW}] \label{thm:afc1}
  Let $X$ be a finite subset of $\kn$, let $f \in \kxn$, and
  let $g \in \I(X) + \submod{f}$ with $g \neq 0$. Then
  \[
  |X \ohne \V (f)| \ge |\cofilter{\I (X)} \, \cap \, \filter{ \{\LM (g) \} }|
  \, .
  \]
\end{thm}
\begin{proof}
  Since $g \in \I(X) + \submod{f}$, $g$ vanishes on $\V(f) \cap X$,
  and thus $X \cap \V(f) \subseteq X \cap \V(g)$, and therefore
  $X \ohne \V(f) \supseteq X \ohne \V(g)$.
  Now by Theorem~\ref{thm:cf},
  we have $|X \ohne \V (g)| =
  | \cofilter{\I (X)} \ohne \cofilter{\I(X) + \submod{g}}|$.
  In addition, $\cofilter{\I (X)} \ohne \cofilter{\I(X) + \submod{g}} =
  \cofilter{\I (X)} \cap \filter{(\I(X) + \submod{g})}
  \supseteq
  \cofilter{\I(X)} \cap \filter{\{g\} } =
  \cofilter{\I(X)} \cap \filter{\{ \LM(g) \} }$.
  Altogether, $|X \ohne \V(f)| \ge
  |\cofilter{\I(X)} \cap \filter{\{ \LM(g) \} }|$.
\end{proof}

\begin{thm}[Alon-F\"uredi-Clark for punctured grids] \label{tm:afcp}
  Let $X = \bigtimes_{i=1}^n X_i, Y = \bigtimes_{i=1}^n Y_i$ be grids over the
   field $\kk$ with $Y_i \subseteq X_i$ for all $i \in \ul{n}$, let
   $P := X \setminus Y$, let
  let $f \in \kxn$, and
  for $i \in \ul{n}$, let
  $a_i := |X_i|$ and $b_i := |Y_i|$, and let $\lea$ be
  an admissible monomial ordering.
  Let $g \in \I(P) + \submod{f}$ with
  $\LM (g) = x_1^{e_1} \cdots x_n^{e_n}$ and $e_i < a_i$ for
  all $i \in \ul{n}$. Then
  \[
   |P \ohne \V (f)| \ge \prod_{i=1}^n (a_i - e_i)
  - \prod_{i = 1}^n \min (b_i, a_i - e_i).
  \]
\end{thm}
\begin{proof}
  By Theorem~\ref{thm:afc1}, we have
  \[
  |P \ohne \V(f)| \ge |\cofilter{\I (P)} \cap \filter{ \{ \LM(g) \} }|.
  \]
  By Lemma~\ref{lem:punctured}, we have
  $\cofilter{\I (P)} \cap \filter{ \{ \LM(g) \} }
  = \{ x^{\alpha} \mid \alpha \in E \}$,
  where $E$ is given by
  \[
    E = (\bigtimes_{i \in \ul{n}} [0, a_i) \setminus
         \bigtimes_{i \in \ul{n}} [a_i - b_i, a_i) )
           \cap
           \bigtimes_{i \in \ul{n}} [e_i, a_i).
  \]
  Then we have
  \begin{multline*}
     E =   \bigtimes_{i \in \ul{n}} [e_i, a_i) \ohne
           \bigtimes_{i \in \ul{n}} [a_i - b_i, a_i) )
       =   \bigtimes_{i \in \ul{n}} [e_i, a_i) \ohne
         (\bigtimes_{i \in \ul{n}} [e_i, a_i) \cap
           \bigtimes_{i \in \ul{n}} [a_i - b_i, a_i) ) \\
       =  \bigtimes_{i \in \ul{n}} [e_i, a_i) \ohne
         \bigtimes_{i \in \ul{n}} [\max (e_i, a_i - b_i), a_i).
           \end{multline*}
    Since $a_i - \max (e_i, a_i - b_i) = \min (a-e_i, b_i)$,       
    we can now compute $|E| = \prod_{i=1}^n (a_i - e_i) -
    \prod_{i=1}^n \min (b_i, a_i - e_i)$.
\end{proof}
We note that similar to Clark's version for unpunctured grids in \cite{Cl:ATCW},
this lower bound can be attained for every choice of $(e_1, \ldots, e_n)$
with $e_i < a_i$ for all $i \in \ul{n}$.
\begin{pro}
    Let $X = \bigtimes_{i=1}^n X_i, Y = \bigtimes_{i=1}^n Y_i$ be grids over the
   field $\kk$ with $Y_i \subseteq X_i$ for all $i \in \ul{n}$, and let
   $P := X \setminus Y$.
     For each $i \in \ul{n}$, let
     $a_i := |X_i|$, $b_i := |Y_i|$, let
     $e_i \in [0, a_i - 1)$ and let $E_i \subseteq X_i$ be such that
       $|E_i| = e_i$ and
       ($E_i \subseteq X_i \setminus Y_i$ or
        $X_i \setminus Y_i \subseteq E_i$).
       Let $f_i := \prod_{a \in E_i} (x_i - a)$ and
       $f := \prod_{i=1}^n f_i$.
       Then for every admissible monomial ordering,
       $\LM (f)= x_1^{e_1} \cdots x_n^{e_n}$, and
       $|P \setminus \V(f)| =  \prod_{i=1}^n (a_i - e_i)
       - \prod_{i = 1}^n \min (b_i, a_i - e_i)$.
\end{pro}
\begin{proof}
  It is easy to see that $\LM(f) = x_1^{e_1} \cdots x_n^{e_n}$ and
  that the nonzeros of $f$ on $X$
  are given by
  $X \setminus \V(f) = \bigtimes_{i=1}^n (X_i \ohne E_i)$,
  and hence $|X \setminus \V(f)| = \prod_{i=1}^n (a_i - e_i)$.
  We will now compute how many of these nonzeros lie in $Y$.
  To this end, we observe that
  \[
  \bigtimes_{i=1}^n (X_i \ohne E_i) \cap \bigtimes_{i=1}^n Y_i =
  \bigtimes_{i=1}^n (Y_i \ohne E_i).
  \]
  In the case that $E_i \subseteq X_i \setminus Y_i$,
  we have $Y_i \ohne E_i = Y_i$, and therefore
  $|Y_i \ohne E_i| = b_i$. In the case that
  $X_i \setminus Y_i \subseteq E_i$, we have
  $X_i \ohne E_i = (Y_i \union (X_i \setminus Y_i)) \ohne E_i
  = (Y_i \ohne E_i) \union ((X_i \ohne Y_i) \ohne E_i) =
  (Y_i \ohne E_i) \union \emptyset = Y_i \ohne E_i$,
  and therefore $|Y_i \ohne E_i| = a_i - e_i$.
  In each of these cases, $|Y_i \ohne E_i| = \min (a_i - e_i, b_i)$.
  Therefore, $|P \ohne \V(f)| = |X \ohne \V(f)| -
              |Y \ohne \V(f)| =  \prod_{i=1}^n (a_i - e_i)
       - \prod_{i = 1}^n \min (b_i, a_i - e_i)$.
\end{proof}

Based on Theorem~\ref{tm:afcp},
we can now prove Theorem~\ref{thm:AFC-punct}.
\begin{proof}[Proof of Theorem~\ref{thm:AFC-punct}]
    We fix an admissible ordering $\lea$ such that
  $\sum_{i=1}^n \gamma_i < \sum_{i=1}^n \delta_i$ implies
  $\mono{\gamma} \lea \mono{\delta}$.

  For proving~\eqref{it:bin1},  
 we divide $f$ by $G'$, where $G'$ is the Gr\"obner basis
  of $\I (P)$ given in Theorem~\ref{thm:ballcn1} and obtain
  a standard expression $f = \sum_{i=1}^{n+1} h_i g'_{i} + r$ such
  that the remainder $r$ contains no monomial divisible
  by a $\LM(g')$ with $g' \in G'$. If $r = 0$, then
  $f \in \I(P)$, and thus $\I(P) \setminus \V(f) = \emptyset$.
  If $r \neq 0$, then $\deg (r) \le \deg (f)$.
  Let $(e_1, \ldots, e_n) := \LEXP (r)$.
  Now Theorem~\ref{tm:afcp} for $g := r$ yields
  $|P \ohne \V(f)| \ge \prod_{i=1}^n (a_i - e_i)
  - \prod_{i = 1}^n \min (b_i, a_i - e_i)$.
  Set
  $y_i := a_i - e_i$. Then $1 \le y_i \le a_i$.
  If $y_i \le b_i$ for all $i \in \ul{n}$, then
  $a_i - b_i \le e_i$ for all $i \in \ul{n}$, and
  thus $\LM (r)$ is divisible by the leading monomial
  $\prod_{i \in \ul{n}} x_i^{a_i - b_i}$ of one of the elements
  of $G'$.  Hence there is $i \in \ul{n}$ with
  $y_i > b_i$.
    Now $\sum_{i=1}^n y_i = \sum_{i=1}^n (a_i - e_i)
     = \sum_{i=1}^n a_i - \sum_{i=1}^n e_i =
       \sum_{i=1}^n a_i - \deg (r) \ge \sum_{i=1}^n a_i - \deg (f)$.
    Hence $(y_1, \ldots, y_n) \in A$, which proves
   \eqref{eq:Zf}.
    
   For proving~\eqref{it:bin2}, 
   let $(e_1, \ldots, e_n) := \LEXP (f)$.
  Now Theorem~\ref{tm:afcp} for $g := f$ yields
  $|P \ohne \V(f)| \ge \prod_{i=1}^n (a_i - e_i)
  - \prod_{i = 1}^n \min (b_i, a_i - e_i)$. Set
  $y_i := a_i - e_i$. Then since $e_i \le \deg_{x_i} (f)$,
  we have $a_i - \deg_{x_i} (f) \le y_i \le a_i$.
    Now $\sum_{i=1}^n y_i = \sum_{i=1}^n (a_i - e_i)
     = \sum_{i=1}^n a_i - \sum_{i=1}^n e_i =
       \sum_{i=1}^n a_i - \deg (f)$.
    Hence $(y_1, \ldots, y_n) \in B$, which proves
   \eqref{eq:Zf2}.
\end{proof}
Setting $Y_i := \emptyset$ for all $i \in \ul{n}$, one recovers
the classical Alon-F\"uredi Theorem for grids.
\section*{Acknowledgements}
The authors thank Pete\ L.\ Clark for sharing \cite{Cl:ATCW}. The second
author wishes to thank the Institute for Algebra at the Johannes
Kepler University Linz for their hospitality in the spring
of 2025 when this work commenced.
\bibliography{cnps19}
\end{document}